\documentclass[10pt]{article}
\usepackage{amsmath,amssymb,amsthm,amsfonts,amssymb,cite, hyperref,enumerate,yfonts,flexisym,bbm}
\usepackage{graphicx,color,soul} 

\def\spine{1.1in}
\usepackage[heightrounded, twoside, top=1in, bottom=1.3in, inner=\spine, outer=\spine]{geometry}

\def\Lim#1{\underset{#1}{\text{\rm Lim\,}}}
\def\B#1#2{B_{#1}(#2)}

\newcommand{\m}{\mathcal{M}}
\def\P{\mathcal{P}}
\def\leq{\leqslant}
\def\geq{\geqslant}
\newcommand{\bs}[1]{\boldsymbol{#1}} 

\def\cf{c_{{}_{\rm Fro}} }
\def\weakto{\stackrel{*}{\scalebox{1.8}[1.0]{$ \rightharpoonup $}} } 

\def\sli{\sum\limits}
\def\ili{\int\limits}
\def\li{\liminf_{N\to \infty} \frac{\P_s^*(A, N)}{N^{s/d}}}
\def\ls{\limsup_{N\to \infty} \frac{\P_s^*(A, N)}{N^{s/d}}}

\def\R{\mathbb{R}}
\def\ep{\varepsilon}
\def\phi{\varphi}
\def\epsilon{\varepsilon}

\def\H{\mathcal{H}}
\def\h{\mathcal{H}}

\def\f{{F}}
\def\t{\mathfrak{t}}
\def\L{\mathcal{H}}
\def\sgn{\text{sgn\,}}

\newcommand{\diam}{\operatorname{diam}}

\newcommand{\dist}{\operatorname{dist}}

\newcommand{\PP}{\mathcal{P}}

\newcommand{\wt}{\widetilde}

\newtheorem{theorem}{Theorem}
\newtheorem{lemma}[theorem]{Lemma}

\theoremstyle{definition} 
\newtheorem{corollary}[theorem]{Corollary}

\newtheorem*{example*}{Example}
\newtheorem{defin}[theorem]{Definition}
\newtheorem{zamech}[theorem]{Remark}
\newtheorem{remark}[theorem]{Remark}

\newtheorem{innercustomthm}{Theorem}
\newenvironment{customthm}[1]
{\renewcommand\theinnercustomthm{#1}\innercustomthm}
{\endinnercustomthm}

\makeatletter
\renewenvironment{proof}[1][\proofname]{\par
  \pushQED{\qed}%
  \normalfont \topsep6\p@\@plus6\p@\relax
  \trivlist
  \item[\hskip\labelsep
        \itshape
    \textbf{\textit{#1}}\@addpunct{.}]\ignorespaces
}{%
  \popQED\endtrivlist\@endpefalse
}
\providecommand{\proofname}{Proof}
\makeatother

\author{A. Anderson
    \and
    A. Reznikov
    \and
    O. Vlasiuk
    \and
    E. White
}

\begin{document}
\title{Polarization and covering on sets of low smoothness} 

\date{\today}
\maketitle 
\begin{abstract}
    In this paper we study the asymptotic properties of point configurations that achieve optimal covering of sets lacking smoothness.  Our results include the proofs of existence of asymptotics of best covering and maximal polarization for $ (\h_d,d) $-rectifiable sets and maximal polarization on self-similar fractals.
\end{abstract}

\vspace{4mm}

\vspace{2mm}



\section{Introduction}
The question of covering a compact set by metric balls or, more generally, by convex bodies is a classical problem in metric geometry, and has multiple important applications. In this paper we focus on covering a set by balls of small radius $\ep$, and deduce the asymptotic behavior for the number of balls needed to cover a given compact set when $\ep \to 0$. Our results allow general norms, so the corresponding balls can in general be rescaled copies of a given bounded convex absorbing balanced set.

The theory of covering by translations of a fixed convex set was developed by Rogers, \cite{Rogersbook, RogersZong}. The early investigation of the asymptotic properties of the optimal covering of compact sets was undertaken in a paper by Kolmogorov \cite{Kolmogorov1956} and subsequently expanded in his joint paper with Tikhomirov \cite{kolmogorovEentropy1959}. Inspired by the work of Shannon in information theory, Hausdorff measures of fractional dimension, and following an earlier study by Pontryagin and Shnirelman~\cite{pontrjaginPropriete1932} on the metric notion of dimension, Kolmogorov defined the so-called $ \epsilon $-entropy of a compact set $ A $ as
\[
    \log_2 N_\epsilon (A),
\]
where $ N_\epsilon(A) $ is the smallest cardinality of an $ \epsilon $-covering of $ A $:
\[
    A\subset \bigcup_{U\in \gamma} U,
\]
with each set $ U\in \gamma $ having a diameter at most $ \ep $. The quantity $ N_\epsilon(A) $ can be understood as the smallest number of points in a discrete quantization of the set $ A $, if the admissible error must be bounded by $ \epsilon $. Accordingly, $ \log_2 N_\epsilon (A) $ is then the ``quantity of information'', measured in bits, contained in this quantization. The aforementioned paper of Kolmogorov-Tikhomirov proceeds to discuss the asymptotic orders of growth of $ N_\epsilon (A) $ for $ \epsilon\downarrow 0 $ in the case of Jordan-measurable compact set $ A $; i.e., a measurable set $A\subset \R^d$ whose boundary has zero Lebesgue measure. They further study metric dimensions of the spaces of functions with finite smoothness and analytical functions.

An alternative approach to the optimal covering problem, and the one used in the sequel, consists in fixing the number of points $ N $ in the quantizer and finding the smallest $ \epsilon_N $, for which set $ A $ is contained in the $ \epsilon_N $-neighborhood of the quantizer. It can be viewed as the question of finding the best quantization of the set $ A $, with the maximal possible error used as objective function that is to be minimized.

Quantization as a part of information theory was actively developed by a number of researchers; the papers by Zador \cite{zadorAsymptotic1982} and Bucklew and Wise \cite{bucklewMultidimensional1982} established the existence of asymptotics on Jordan-measurable sets for a related functional: expected quantization error. The monograph of Graf and Luschgy \cite{grafFoundations2000} summarized these developments both for the asymptotics of quantization error, and for optimal covering. They also rediscovered the results of Kolmogorov and Tikhomirov about asymptotics of covering for Jordan-measurable sets. Graf and Luschgy then conjectured that the asymptotics for $ N \to \infty $ must exist for sets, more general than just Jordan-measurable. This fact will indeed be one of our main results, Theorem~\ref{thm:covering_asmpt}.

We note that the Jordan measurability assumption is in effect a smoothness condition for the set $A$. In this paper we study the asymptotic properties of point configurations that achieve optimal covering for embedded sets lacking such smoothness properties.  Our results include the proofs of existence of asymptotics of best covering for $ (\h_d,d)  $-rectifiable sets.
We also show existence of the asymptotics for the polarization $ \mathcal P_s^* $ with hypersingular Riesz kernel on $ (\h_d,d)  $-rectifiable sets, as well as the necessary and sufficient condition on fractal contraction ratios, under which the polarization asymptotics exist on self-similar fractals satisfying the open set condition. Analogous results for covering on fractals were obtained by Lalley~\cite{lalleyPacking1988}.

The study of covering properties of various functional spaces, initiated by Kolmogorov and Tikhomirov, was continued in various settings by a number of authors, including Birman-Solomyak, Bourgain-Pajor-Szarek-Tomczak-Jaegermann, Posner-Rodemich-Rumsey, Temlyakov, and others \cite{birmanPiecewisepolynomial1967,bourgainDuality1989,posnerEpsilon1967,temlyakovEntropy2017}. Notably, it has been observed that metric entropy is related to the small ball problem \cite{kuelbsMetric1993}. In most cases, the primary goal of these works is to establish the order of growth of the metric entropy for a certain set; in the Euclidean space however we can obtain the existence of a constant in the asymptotics, depending only on the Hausdorff dimension of the set when the dimension is integer. For fractal sets we have the existence of the constant as well, but it generally depends on the set. 

Recently, a general approach for studying first-order asymptotics of interaction functionals has been developed \cite{hardinAsymptotic2021}. It transpires that both the existence of asymptotic and weak$ ^* $-distribution of the minimizers of a functional follows from its so-called short-range properties (which will be further discussed in Section~\ref{sec:proofs_recti}). Intuitively, if the contribution of pairs of nearby points in $ \omega_N $ dominates over the contributions of remote pairs, there necessarily exist asymptotics with respect to a polynomial rate function. In this paper we show that this approach applies to the problems of optimal covering and polarization, and employ it as one of our main tools. 
In a remarkable coincidence, the two papers addressing this phenomenon from the opposite sides: one due to Gruber, on the asymptotics for optimal quantization with general kernels, another due to Hardin and Saff, on the asymptotics of hypersingular Riesz energy, appeared in {\it Advances in Mathematics} within several months of each other \cite{gruberOptimum2004, hardinMinimal2005}. 

While an earlier version of this manuscript was in preparation, we learned that Cohn and Salmon obtained similar results about the asymptotics of covering on $ (\h_d,d) $-rectifiable sets simultaneously \cite{cohnSphere2021}. It is worth noting that their approach is, like ours, based on the fact that the covering functional is a short-range interaction; they also study asymptotics of sandwich functions on graphs using some of the short-range interaction ideas. This shows how short-range interactions can be the right tool for computing asymptotics in a number of contexts.

\section{Preliminaries}

As explained above, we shall be interested in the asymptotics of $ N $-point {\it best covering} of the compact set $ A \subset \mathbb R^p $, defined as
\[
    \rho^*(A,N) = \inf_{\omega_N\subset \mathbb R^p} \sup_{y\in A} \dist (y,\omega_N)
\]
where the distance between a point and a multiset $\omega_N=\{x_1, \ldots, x_N\}$ (note, the possibility $ x_i=x_j $ is allowed) is defined naturally as $\dist (y,\omega_N) = \min_{x\in\omega_N} \|x-y \| $. Here and in the following we write $ \|\cdot \| $ for a fixed norm in $ \mathbb R^p $ (in particular, not necessarily Euclidean). We shall study the behavior of this quantity for $ N\to \infty $. A related quantity is the $ N $-point {\it constrained best covering}:
\[
    \rho(A,N) = \inf_{\omega_N\subset A} \sup_{y\in A} \dist (y,\omega_N),
\]
in which optimization over the configuration occurs on subsets of $ A $ only. Notice that due to compactness of $ A $, $ \inf $ and $ \sup $ can be replaced with $ \min $ and $ \max $ respectively; we shall still need the above definitions for non-compact input sets.

Following \cite{BorodTAMS}, we define the problem of {maximal polarization} with Riesz kernels. Given $ \omega_N $, $ A\subset \mathbb R^p $ as above, let 
$$P_s(\omega_N,A) = \inf_{y\in A} \sum_{x\in\omega_N} \| y- x\|^{-s},$$
the {\it Riesz s-polarization} achieved by configuration $ \omega_N $ on $ A $. {\it Maximal unconstrained polarization} for $ A $ with cardinality $ N $ is then given by
\[
    \mathcal P_s^* (A,N) = \sup_{\omega_N\subset \mathbb R^p} P_s(\omega_N,A).
\]
Similarly to the covering radius case, we may want the configuration $\omega_N$ to stay in the set $A$. In this case, we introduce the {\it maximal constrained polarization} as
\[
    \mathcal P_s (A,N) = \sup_{\omega_N\subset A} P_s(\omega_N,A).
\]
The above quantity is closely related to the covering radius of the set $A$, but has a somewhat less local structure. 
The study of the maximal polarization was initially motivated by the notion of Chebyshev constant in polynomial approximation theory \cite{erdelyiRiesz2013}, and has since been discussed in a number of papers \cite{BorodTAMS, reznikovCovering2018, hardinUnconstrained2020}.

\section{Notation}
Before we proceed to formulate the main results, let us introduce some notation. We consider a fixed norm $ \|\cdot \| $, not necessarily Euclidean, on the ambient spaces $ \mathbb R^d $ and $ \mathbb R^p $, $ p\geq d $. With respect to this norm, we define 
$ \mathcal H_d $, the $ d $-dimensional Hausdorff  measure, normalized so that $ \mathcal H_d([0,1]^d) = 1 $; we then write $ v_d = \L_d(\B{ 1}{\bs 0}) $ for the volume of the $ d $-dimensional unit ball. The $ d $-dimensional Minkowski content is denoted by $ \mathcal M_d $.

Let $ \B{r}{x} $ stand for the closed ball of radius $ r > 0 $ centered at $ x \in \mathbb R^p $, with respect to the norm $ \|\cdot\| $. Similarly, $ \B{r}{A} = \bigcup_{x\in A} \B{r}{x} $ for $ r>0 $ and a compact set $ A $. As discussed in the introduction, we are interested in the quantity
$$ R(\omega_N,A) =\sup_{y\in A} \text{dist}(y,\omega_N), $$
which denotes the smallest radius $ r $ for which $ A\subset B_r(\omega_N) $.

\section{Main results on rectifiable sets}
In this section we establish the existence of the asymptotics of best covering and optimal polarization on compact subsets of $ \mathbb R^d $ of positive $ \mathcal H_d $-measure. This strengthens the previously known asymptotic results for covering and polarization \cite{kolmogorovEentropy1959, hardinUnconstrained2020}.

Recall that for Jordan-measurable sets in $ \mathbb R^d $, the asymptotics of best covering is known to exist and depend only on the measure of the covered set and the dimension $ d $. More precisely, the following result was known to hold.
\begin{customthm}{A}[Kolmogorov-Tikhomirov \cite{kolmogorovEentropy1959}; Graf-Luschgy \cite{grafFoundations2000}]
    \label{thm:graf_luschgy}
    Given a nonempty compact $ A \subset \mathbb R^d $ with $ \L_d(\partial A) = 0 $, one has
    \[
        \lim_{N\to \infty} N^{1/d}\rho^*(A,N) = {\theta_{d}}{\mathcal H_d(A)^{1/d}},
    \]
    with the finite positive constant $ \theta_d $ depending only on the dimension $ d $ and the norm chosen in $ \mathbb R^d $.
\end{customthm}
Concerning maximal discrete polarization on Jordan sets, the following results were obtained for the constrained and unconstrained polarization. 
\begin{customthm}{B}[Borodachov-Hardin-Reznikov-Saff\cite{BorodTAMS}; Hardin-Saff-Petrache\cite{hardinUnconstrained2020}]
    \label{thm:jordanpolar}
Given a nonempty compact $ A \subset \mathbb R^d $ with $ \L_d(\partial A) = 0 $, one has for every $s>d$,
$$
\lim_{N\to\infty} \frac{\PP_s(A, N)}{N^{s/d}} = \lim_{N\to\infty} \frac{\PP^*_s(A, N)}{N^{s/d}} = \frac{\sigma_{s,d}}{\L_d(A)^{s/d}},
$$
where the finite positive constant $\sigma_{s,d}$ depends on $ s $, $ d $, and the used norm only.
\end{customthm}
\begin{remark}
    The results formulated in Theorem~\ref{thm:jordanpolar} in the above references were originally stated for the Euclidean distance, but the same proofs apply to general norms.
\end{remark}
We further remark that in \cite{BorodTAMS}, a similar theorem was proved for $C^1$-smooth embedded manifolds, and later in \cite{hardinUnconstrained2020} the assumption $\L_d(\partial A)=0$ was dropped. However, almost nothing was known for embedded sets of lower smoothness. To state our main results, we need some classical notions of smoothness from geometric measure theory.
\begin{defin}
A compact set $ A\subset \mathbb R^p $ is called {\it $ d $-rectifiable}, if $A=f(K)$ for a compact set $K\subset \R^d$ and a Lipschitz function $f\colon \R^d \to \R^p$. 
A set $ A \subset \mathbb R^p $ is said to be {\it $ (\h_d, d) $-rectifiable}, if $A$ can be written as
$$
A = \bigcup_{j=1}^\infty A_j \cup R,
$$
where each $A_j$ is $d$-rectifiable, $\H_d(R)=0$, and the union in the right-hand side is disjoint.
\end{defin}

In the next definition we introduce the notion of Minkowski content, which will be necessary to state the smoothness assumptions on the set $A$.
\begin{defin}
Recall that for a compact set $A\subset \R^p$, we use the following notation for the $ r $-neighborhood of $ A $:
\begin{equation}\label{eq:Aofepsilon}
\B{ r}{A}:=\{x\in \R^p: \textup{dist}(x, A)\leq  r\}
\end{equation}
Furthermore, for every $d\in (0, p]$ we define
$$
\underline{\mathcal M}_d(A):=\liminf_{ r\downarrow 0} \frac{\L_p(\B{ r}{A})}{v_{p-d} r^{p-d}}, \quad \overline{\mathcal M}_d(A):=\limsup_{ r\downarrow 0} \frac{\L_p(\B r A)}{v_{p-d} r^{p-d}}.
$$
When these limits are equal, their common value will be denoted $\mathcal{M}_d(A)$, the {\it Minkowski content of set $ A $}.
\end{defin}

Existence of $ \m_d(A) $ is equivalent to the existence of box-counting dimension of set $A$. The equality $\mathcal{M}_d(A)=\h_d(A)$ holds for closed subsets of $ \mathbb R^d $ as well as for reasonably good subsets of $ \mathbb R^p $; it will be a part of our smoothness assumptions in the sequel. 
With that in mind, we are ready to state our first main theorems, which significantly generalize results from \cite{BorodTAMS} and \cite{hardinUnconstrained2020}. We will formulate them separately for polarization and covering radius functionals.
\begin{theorem}
    \label{thm:covering_asmpt}
    For any compact $ (\mathcal \h_d,d)  $-rectifiable $ A \subset \mathbb R^p $ satisfying $ \h_d(A) = \m_d(A) $ one has
    \[
        \lim_{N\to\infty}N^{1/d} \rho(A, N)=\lim_{N\to \infty} N^{1/d}\rho^*(A,N) = {\theta_{d}}{\mathcal H_d(A)}^{1/d},
    \]
    with the constant $ \theta_d $ depending only on the dimension $ d $ and the norm chosen in $ \mathbb R^p $. In particular, the above holds for any compact set $ A\subset \mathbb R^d $.
\end{theorem}

\begin{theorem}
    \label{thm:polar_asmpt}
    For any compact $ (\mathcal \h_d,d)  $-rectifiable set $ A \subset \mathbb R^p $ satisfying $ \h_d(A) = \m_d(A) $ one has, for $s>d$,
    \[
        \lim_{N\to \infty} \frac{\P_s(A,N)}{N^{s/d}} = \lim_{N\to \infty} \frac{\P_s^*(A, N)}{N^{s/d}}=\frac{\sigma_{s,d}}{\mathcal H_d(A)^{s/d}},
    \]
    with the constant $ \sigma_{s,d} $ depending only on the dimension $ d $, exponent $ s $, and the norm chosen in $ \mathbb R^p $. In particular, the above holds for any compact set $ A\subset \mathbb R^d $.
\end{theorem}
We further remark that examples of $ (\h_d,d) $-rectifiable sets that satisfy $ \h_d(A) = \m_d(A) $ include $ d $-rectifiable compact sets and bi-Lipschitz images of closed sets in $ \mathbb R^d $. In what follows we say that the map $ \psi: \mathbb R^{d} \to \mathbb R^p $ has {\it bi-Lipschitz constant $ 1+K $}, if there holds 
\[
    (1+K)^{-1} \| x-y\| \leq \| \psi(x)- \psi(y)\| \leq (1+K)\| x-y\| \qquad \text{ for any }x,y\in A \subset \mathbb R^{d}.
\]
Note that this definition applies to norms in spaces $ \mathbb R^d $ and $ \mathbb R^p $, which do not need to agree or be Euclidean, but are fixed throughout the paper (and by an abuse of notation we write $ \|\cdot \| $ for either norm). Observe also that $ \psi $ must necessarily be invertible, with $ \psi^{-1} $ being a bi-Lipschitz map with the same constant.

The covering and polarization functionals share the common property of short-range functionals, by which any sequence of configurations attaining the values of asymptotics from Theorems~\ref{thm:covering_asmpt}--\ref{thm:polar_asmpt} is distributed uniformly over the set $ A $. One can follow the proof of \cite[Theorem 1.14]{hardinUnconstrained2020} to obtain this result; we provide a short independent argument for completeness. 

\begin{theorem}\label{thm:distr}
Assume $ A \subset \mathbb R^p $ is a compact $ (\mathcal \h_d,d)  $-rectifiable set satisfying $ \h_d(A) = \m_d(A)>0 $. Let $\{\omega_N\}_{N=1}^\infty$ be any sequence of configurations in $\R^p$ such that
$$
\lim_{N\to \infty} N^{1/d}R(\omega_N,A) = {\theta_{d}}{\mathcal H_d(A)}^{1/d}, \;\;\;\; \text{or} \;\;\;\;  \lim_{N\to \infty} \frac{P_s(A, \omega_N)}{N^{s/d}}=\frac{\sigma_{s,d}}{\mathcal H_d(A)^{s/d}}.
$$
Then the sequence of corresponding empirical measures satisfies
$$
\frac{1}{N} \sum_{x\in \omega_N} \delta_x \weakto \frac{\H_d(\cdot \cap A)}{\H_d(A)},
$$
with the convergence understood in the weak$^*$ sense.
\end{theorem}

We also show that the smoothness assumptions for set $ A $ in Theorems~\ref{thm:covering_asmpt}--\ref{thm:polar_asmpt} are sharp in that the expression for the asymptotics in terms of $ \h_d(A) $ no longer holds without them.
\begin{theorem}\label{thm:minkisbigger}
Assume $A\subset \R^p$ is a compact $ (\mathcal \h_d,d)  $-rectifiable set satisfying $ \h_d(A) < \overline{\m}_d(A)$. Then we have
$$
\limsup_{N\to \infty}N^{1/d} \rho(A, N) \geqslant \limsup_{N\to \infty}N^{1/d} \rho^*(A, N) > \theta_d \H_d(A)^{1/d},
$$
and, for sufficiently large $s$, 
\[
        \liminf_{N\to \infty} \frac{\P_s(A,N)}{N^{s/d}} \leqslant \liminf_{N\to \infty} \frac{\P_s^*(A, N)}{N^{s/d}}<\frac{\sigma_{s,d}}{\mathcal H_d(A)^{s/d}}.
    \]
\end{theorem}
Finally, to conclude this section, we present some general estimates on covering and polarization.
\begin{theorem}\label{thm:generalest}
Assume $A\subset \R^p$ is a compact set. Then, for some positive finite constants $c_1$ and $c_2$, that depend on $d$ and $p$, we have
$$
c_1 \underline{\m}_d(A) ^{1/d} \leqslant \liminf_{N\to \infty}N^{1/d} \rho^*(A, N) \leqslant \liminf_{N\to \infty}N^{1/d} \rho(A, N)\leqslant c_2 \underline{\m}_d(A) ^{1/d},
$$
$$
c_3 \overline{\m}_d(A) ^{1/d} \leqslant \limsup_{N\to \infty}N^{1/d} \rho^*(A, N) \leqslant \limsup_{N\to \infty}N^{1/d} \rho(A, N)\leqslant c_4 \overline{\m}_d(A) ^{1/d}.
$$
\end{theorem}
One outcome of this theorem is a corollary that illustrates the sharpness of our assumption of existence of $\m_d(A)$.
\begin{corollary}
There exists a finite positive number $C$ that depends on $d$ and $p$ with the following property: if $A\subset \R^p$ is a compact set, and $\overline{\m}_d(A)>C\underline{\m}_d(A)$, then the limits
$$
\lim_{N\to \infty}N^{1/d} \rho^*(A, N) \;\;\;\; \text{and} \;\;\;\; \lim_{N\to \infty}N^{1/d} \rho(A, N)
$$
do not exist.
Furthermore, if $\overline{\m}_d(A)>C\underline{\m}_d(A)$, then for sufficiently large $s$, the limit
$$
\lim_{N\to \infty}\frac{\P^*_s(A, N)}{N^{s/d}}
$$
does not exist.
\end{corollary}
\section{Covering using optimal polarization sets}
This section quantifies the connection between the covering radius and polarization; namely, we will show that the discrete polarization resembles the covering radius when $s\to \infty$ in a strong sense. Moreover, we will show some local distributional properties of the optimal configurations for $\P_s^*(A, N)$, which did not follow from global properties in Theorem \ref{thm:distr}. We begin with the following general estimate.

\begin{theorem}
    \label{thm:covering_polarized}
    Let $A\subset \R^p$ be a compact set, $\H_d(A)>0$, and $s>p$. Then for any configuration $\omega_N$ that attains $\P^*_s(A, N)$ we have
    \begin{equation*}
        \rho^*(A, \omega_N) \leqslant C(s,d,p,A) N^{\frac{p}{d(s-p)}} \left(\P_s^*(A, N)\right)^{\frac{1}{p-s}}.
    \end{equation*}
    Moreover, 
    $$
    \lim_{s\to \infty} C(s,d,p,A)=1.
    $$
\end{theorem}
This theorem has several corollaries. The first one concerns the limit of polarization as $s\to \infty$.
\begin{corollary}
    \label{cor:limiting_polar}
    Let $A\subset \R^p$ be a compact set with $\H_d(A)>0$. Then
    \begin{equation}\label{sasha_limits1}
        \lim_{s\to \infty} \left(\li\right)^{1/s} = \frac{1}{\limsup\limits_{N\to\infty} \rho^*(A, N)N^{1/d}},
    \end{equation}
    and
    \begin{equation}\label{sasha_limits2}
        \lim_{s\to \infty} \left(\ls\right)^{1/s} = \frac{1}{\liminf\limits_{N\to\infty} \rho^*(A, N)N^{1/d}}.
    \end{equation}
\end{corollary}
This corollary is useful, in particular, in proving the non-existence of the asymptotics. Indeed, the following statement clearly follows from the corollary above.
\begin{corollary} \label{cor:nonex}
Let $A\subset \R^p$ be a compact set with $\H_d(A)>0$. Assume the limit for the covering radius does not exist; i.e.,
$$
\limsup\limits_{N\to\infty} \rho^*(A, N)N^{1/d} > \liminf\limits_{N\to\infty} \rho^*(A, N)N^{1/d}.
$$
Then, for sufficiently large values of $s$, we have
$$
\ls > \li.
$$
\end{corollary}
Our next corollary gives the covering properties of optimal configurations for $\P_s^*(A,N)$ with minimal assumptions on the set $A$. We note that the order of covering that we establish in the statement below is optimal, since for sets of positive $ \h_d $-measure, optimal covering has asymptotic order at least $N^{1/d}$. For convex sets, the estimate below was proved in \cite{reznikovCovering2018}.
\begin{corollary}
    \label{cor:cover_polar}
    Let $A\subset \R^p$ be a compact set and assume that for some finite positive constant $C_s$ we have, for every $N\geqslant 1$, an estimate $\P_s^*(A, N)\geqslant C_s N^{s/d}$. Then there exists a finite positive constant $R_s$ such that for any $N\geqslant 1$ and any configuration $\omega_N$ optimal for $\P_s^*(A, N)$, we have
    $$
		R(\omega_N,A) \leqslant R_s N^{1/d}.
		$$
In particular, our assumption is satisfied if $\overline{\mathcal{M}}_d(A)<\infty$.
\end{corollary}

\section{Main results on fractal sets}
Recall \cite{Hutchinson} that a \textit{similitude contraction} $ \psi: \mathbb R^p \to \mathbb R^p $ can be written as
\[
    \psi( x) = r\cdot O x +  z
\]
with an orthogonal matrix $ O \in \mathcal O(p) $, a vector $  z \in\mathbb R^p $, and a contraction ratio $ 0 <r< 1 $.  
It is well-known that any collection of similitude contractions has a compact set of fixed points. Conversely, it will be convenient to consider the class of fractals defined as the set of fixed points of similitude contractions.
\begin{defin}
A compact set $ A\subset \mathbb R^p $ is a \textit{self-similar fractal} with similitudes $ \left\{\psi_m\right\}_{m=1}^M $ with contraction ratios $ 0< r_m < 1$,  $ 1\leq m\leq M $, if
\[
    A = \bigcup_{m=1}^M \psi_m(A),
\]
where the union is disjoint.
We further say that $A$ satisfies the \textit{open set condition} if there exists a bounded open set $ V\subset \mathbb R^p $ such that 
\[
    \bigcup_{m=1}^M \psi_m(V) \subset V,
\]
where the union in the left-hand side is disjoint.
\end{defin}
In what follows the fractal sets are always assumed to be self-similar and satisfying the open set condition. For such an $A$, it is known \cite{MR867284, MR1333890} that its Hausdorff dimension $\dim_H A = d$ is the unique solution of the equation
\begin{equation}
    \label{eq:dim}
    \sum_{m=1}^M r_m^d =1.  
\end{equation}
It will also be used that such fractal sets $ A $ are {\it $d$-regular}, that is, satisfying
\begin{equation}
    \label{eq:dregular}
    c^{-1} r^d \leq \H_d(A\cap \B{ r}{ x}) \leq c r^d,
\end{equation}
for any $ x\in A $ and $ 0 \leq r \leq \diam A $, with some constant $ c>0 $. It is well-known \cite{hardinUnconstrained2020},\cite{erdelyiRiesz2013}, that $ d $-regularity implies that for some constant $C>0$ and any integer $N\geqslant 1$ we have
\begin{equation}
    \label{eq:asymp_bounds}
    C^{-1} \leq \frac{\PP_s^*(A, N)}{N^{s/d}} \leq C
\end{equation}
The same estimates hold for the constrained polarization $ \mathcal P_s $; the constant $ C $ in general depends on $ A $.
It was noticed by Lalley \cite{lalleyacta,lalleyPacking1988} that the existence of the limit
$$
\lim_{N\to \infty}N^{1/d}\rho^*(A, N)
$$
can be tackled using renewal theory, and thus depends on the properties of the multiplicative subgroup of $ \mathbb R $, generated by $\{r_m\}_{m=1}^M$ \cite{feller}. Our first main result of this section is to prove the same result for polarization, which is a less local and, therefore, harder to handle quantity. 
\begin{theorem}\label{sasha_fractal_indep}
Let $A$ be a fractal set defined by similitudes $\{\psi_m\}_{m=1}^M$ with contraction ratios $\{r_m\}_{m=1}^M$. Let $s>d=\dim_H(A)$.
If the set
$$
\{t_1 \log(r_1) + \cdots + t_M\log(r_M): t_1,\ldots,t_M\in \mathbb{Z}\}
$$
is dense in $\R$, then the limits
$$
\lim_{N\to \infty} \frac{\PP_s^*(A, N)}{N^{s/d}},\qquad
\lim_{N\to \infty} \frac{\PP_s(A, N)}{N^{s/d}}
$$
exist.
\end{theorem}
On the other hand, when the set of linear combinations of the contraction ratios is not dense in $\R$, Lalley proves existence of the limits over certain implicit sub-sequences. However, nothing is said about the sharpness of this result. We follow the ideas from \cite{andrez} to show that the assumption of Theorem \ref{sasha_fractal_indep} is sharp.
\begin{theorem}\label{sasha_nonex}
Let $A$ be a fractal set defined by similitudes $\{\psi_m\}_{m=1}^M$ with contraction ratios $\{r_m\}_{m=1}^M$. Let $d=\dim_H(A)$ and $s>d$.
If the set
$$
\{t_1 \log(r_1) + \cdots + t_M\log(r_M): t_1,\ldots,t_M\in \mathbb{Z}\}
$$
is not dense in $\R$, then the limits
$$
\lim_{N\to\infty} \rho(A, N)N^{1/d}, \quad \text{and} \quad \lim_{N\to\infty} \rho^*(A, N)N^{1/d}
$$
do not exist. Therefore, Corollary \ref{cor:nonex} implies that, for large values of $s$, the limit
$$
\lim_{N\to \infty} \frac{\PP_s^*(A, N)}{N^{s/d}}
$$
does not exist.
\end{theorem}

\section{Proofs}

\subsection{Polarization and covering on d-rectifiable sets}
\label{sec:proofs_recti}
We introduce some notation to prove Theorems \ref{thm:covering_asmpt} and \ref{thm:polar_asmpt} in a unified way. 
In order to deal with polarization and covering simultaneously, observe that both are functionals of the form $ f(\omega_N, A) $, taking a multiset of cardinality $ N $, $ \omega_N\subset \mathbb R^p $, and a compact set $ A\subset \mathbb R^p $ as inputs.
We can distinguish the two by the sign exponent in the corresponding rate function $ \t(N) = N^\sigma $, namely,
\[
    \begin{aligned}
        f^-(\omega_N,A) &:=  R(\omega_N,A),\qquad   \sigma = -1/d, \\
        f^+(\omega_N,A) &:=P_s(\omega_N,A),\qquad   \sigma = s/d.
    \end{aligned}
\]
Notice that the dependence of $ f $ on $ s $ and $ d $ is suppressed for brevity. These parameters will be fixed in each specific argument, so it will not lead to confusion; we always however assume $ s > d $.
The result of optimizing such $ f $ over $ N $-point multisets in $ \mathbb R^p $ will be denoted by $ \f^\pm(A,N)  $:
\[
    \begin{aligned}
        F^-(\omega_N,A) &:=\rho^*(A,N), \qquad  \sigma = -1/d, \\
        F^+(\omega_N,A) &:=\P_s^*(A,N), \qquad  \sigma = s/d.
    \end{aligned}
\]

Our strategy is a modification of that in \cite{hardinAsymptotic2021}:
we intend to show, in effect, that the functionals $ \rho^*(\omega_N,A) $ and $ \P_s^*(A,N) $ are so-called {\it short-range interactions} \cite{hardinAsymptotic2021} with rates  $ \t(N) = N^\sigma $ for $ \sigma \in \{ -1/d, s/d \} $, respectively.
To that end we will assume that functionals $ \f^\pm $ satisfy the following axioms and obtain the desired asymptotic behavior. Verification of the axioms is deferred to Section~\ref{sec:verification}. 
To state the axioms, we will need an additional piece of notation for the binary $ \min $ and $ \max operators $:
\[
    \bs\circ = \begin{cases}
        \vee, \text{ binary }\max,   &  \sigma = -1/d, \\
        \wedge, \text{ binary }\min, &  \sigma = s/d.
    \end{cases}
\]
We write $ \sgn \sigma \in \{ \pm 1 \} $ to denote the sign of the exponent $ \sigma $.
\begin{itemize}
    \item \ul{Monotonicity}: for compact sets $ A \subset B \subset \mathbb R^p $ and $ N\geq 1 $ there must hold
        \begin{equation}
            \label{eq:monotonicity}
            \sgn\sigma \cdot \f^\pm(A, N) \geqslant \sgn\sigma \cdot \f^\pm(B, N).
        \end{equation}
    \item \ul{Asymptotics on Jordan-measurable sets}: asymptotics of the considered functionals on Jordan-measurable sets $  A\subset \mathbb R^d $  exist and depend only on the volume of the set:
        \[
            \lim_{N\to \infty} \frac{\f^\pm(A,N)}{N^\sigma} = \frac{c(\f^\pm)}{\L_d(A)^\sigma}.
        \]
        In particular, the constant $ c(\f^\pm) $ is equal to the value of these asymptotics on the unit cube $ [0,1]^d $.
    \item \ul{Short-range property}: asymptotics of the functional $ \f^\pm$ over  unions of sets that are positive distance apart.
		
        Assume that the sets $A_1,A_2\subset \R^p$ are compact, $\text{dist}(A_1, A_2)=2h$ for some $h>0$, and $\overline{\m}_d(A_m)<\infty$, $ m=1,2 $. 

    {\bf (i)} For any $N$-point configuration $\omega_N$ such that $\f^\pm(A_1\cup A_2, N)=f(\omega_N, A_1\cup A_2)$, let
		$$
        N_1:=\#(\omega_N \cap \B h {A_1}), \;\;\;\;\; N_2:=\#(\omega_N \cap \B h {A_2}),
		$$
		where $\B{h}{A_m}$ is the $ h $-neighborhood of $ A_m $, as in \eqref{eq:Aofepsilon}. Then
        \begin{equation}\label{eq:stab_1}
        \liminf_{N\to\infty} \sgn \sigma \cdot \frac{\f^\pm(A_1, N_1) \bs{\circ} \f^\pm(A_2, N_2) }{\f^\pm (A_1\cup A_2, N)} \geq \sgn \sigma.
    \end{equation}
    {\bf (ii)}
    For any two sequences of positive integers $N_1$ and $N_2$ such that $N_1, N_2 \to \infty$, we have:
    \begin{equation}\label{eq:stab_2}
        \limsup_{N\to\infty} \sgn \sigma \cdot \frac{\f^\pm(A_1, N_1) \bs{\circ} \f^\pm(A_2, N_2) }{\f^\pm(A_1\cup A_2, N_1+N_2)} \leq \sgn \sigma.
    \end{equation}

    Notice that the short-range property extends to any finite number of disjoint compact sets $ A_m $, $ 1\leq m \leq M $, by induction.
\item \ul{Stability}: asymptotics of $\f^\pm(A,N)$ for $ N\to \infty $ are stable under perturbations of the set $ A $ with small changes of the Minkowski content.
        More precisely, for every $ (\h_d,d) $-rectifiable compact set $ A\subset \mathbb R^p $ with $0 < \m_d(A)=\h_d(A) < \infty$ and a given $ \epsilon >0 $, there exists $ \delta = \delta(\epsilon, s,p,d,A) $ such that 
    \begin{equation}
        \label{eq:stability}
        \sgn \sigma \cdot \Lim{N\to\infty} \frac{ \f^\pm(A,N)}{N^\sigma} \geq  
         (\sgn \sigma - \epsilon) \cdot \Lim{N\to\infty} \frac{\f^\pm(D,N)}{N^\sigma}\qquad \text{for\ Lim}\in\{ \liminf,\ \limsup \},
    \end{equation}
    whenever the compact $ D \subset A $ satisfies $ \h_d(D) > (1 - \delta )\,\h_d(A)$. While it is not necessary in general, we shall further use that for $p=d$, $ \delta $ can be chosen independently of $ A $, since this holds for both $ \f^\pm $ and will simplify our proofs.
\end{itemize}
We first prove Theorems \ref{thm:covering_asmpt} and \ref{thm:polar_asmpt} assuming the above axioms hold. We start with the following lemma.
\begin{lemma}
    \label{lem:small_covering}
    Suppose $ A\subset \mathbb R^p $ is a compact set, $d\leqslant p$, and $ 0 < \underline{\m}_d(A) \leq \overline{\m}_d(A) < \infty $. Then for $ N \geq N_0 $,
    \begin{itemize}
        \item $\rho^*(A,N) \leq C_1 N^{-1/d}$;
        \item $\P_s^*(A,N) \geq C_2 N^{s/d}$ 
        \item $\rho^*(A,N) \geq C_3 N^{-1/d}$.
    \end{itemize}
    with constants $ N_0 $ and $ C_i $ depending on $ s,p,d,A $ only, $ 1\leq i \leq 3 $.
\end{lemma}
\begin{proof}
    Fix an $ \epsilon > 0 $.
    Let $ M > 0 $ and a sequence $ r_n\downarrow 0 $, $ n\geq 1 $, be such that 
    $$ M = \lim_{n\to \infty} \L_p( \B{r_n}{A} ) / v_{p-d}{r_n}^{p-d}. $$
    Then, for some $n_0=n_0(\epsilon, p,d,A)$ we have, for any $ n > n_0$,
    \[
        \L_p(\B{r_n}{A}) \leq (M + \epsilon) v_{p-d}{r_n}^{p-d}.
    \]
    For every $n > n_0$, consider the maximal $r_n$-separated subset $ \omega_k $ in $ A $.  
    Given $ x,y\in \omega_k $, we have $ \mathcal L_p [\B{r_n/2}{x} \cap \B{r_n/2}{y}] =0 $, and $\B{ r/2}{x} \subset \B{r/2}{A}$, implying, for $ n > n_0$,
    \[
        k\leqslant \frac{\sli_{x\in \omega_k} \L_p(\B{{r_n}/2}{x})}{v_p ({r_n}/2)^p} = \frac{\L_p\left( \bigcup_{x\in \omega_k} \B{ {r_n}/2}{x} \right)}{v_p ({r_n}/2)^p} \leqslant \frac{\L_p(\B{{r_n}/2}{A})}{v_p ({r_n}/2)^p}\leqslant c(p,d) \left(M+\epsilon\right) {r_n}^{-d}.
    \]
    Now pick $ N=N(n) $ so that 
    \[
        \left(\frac{N}{c(p,d) (M+\epsilon)}\right)^{-1/d} \leq r_n < \left(\frac{N-1}{c(p,d) (M+\epsilon)}\right)^{-1/d}.
    \]
    By the above discussion, for $ n > n_0 $ there holds $k\leqslant N$; in addition, the maximality of $\omega_k$ gives that 
    $$
    A\subset \bigcup_{x\in \omega_k} \B{r_n}{x}.
    $$
    Thus, we covered $A$ by $k\leq N$ balls of radius $r_n$, so that there holds $\rho^*(A, N)\leqslant \rho^*(A, k)\leqslant r_n \leq c(\epsilon,p,d,M) (N-1)^{-1/d}$, and the first part of the lemma follows by setting e.g. $ \epsilon = 1 $ and $ M = \overline{\m}_d(A) $. To prove the second part of the lemma, we notice that for $ N $ as above,
    $$
    \P^*_s(A, N)\geqslant \P_s^*(A, k) \geqslant P_s(\omega_k, A) \geq r_n^{-s} = C_2 N^{s/d},
    $$	
    where the last inequality follows from the fact that $\omega_n$ covers $A$ with balls of radius $r$.

    For the third part, again consider $ \underline{\m}_d(A) \leq  M \leq \overline{\m}_d(A) $ and sequence $ \{ r_n \} $ as above; recall that  $ \underline{\m}_d(A) > 0 $ and fix $ 0< \epsilon < \underline{\m}_d(A) $.
    Assume again that for some $n_0 = n_0(\epsilon,p,d,A)$, whenever $ n> n_0$, there holds
    \[
        \L_p(\B{{r_n}}{A}) \geq (M-\epsilon)  v_{p-d}{r_n}^{p-d}.
    \]
    Choose $ N = N(n) $ so that 
    \[
        \rho^*(A,N) \leq r_n < \rho^*(A,N-1).
    \]
    Denote $ \rho = \rho^*(A,N) $ and pick a configuration $ \omega_N $ attaining the covering radius $ \rho $; since $ A \subset \B {\rho}{\omega_N} \subset \B {r_n}{\omega_N} $, there holds $  \B {r_n} A \subset \B {2r_n}{\omega_N} $ and
    \[
        N\cdot 2^pv_p \cdot r_n^p \geq \L_p(\B{r_n}{A}) \geq (M-\epsilon)  v_{p-d}r_n^{p-d}, \qquad n \geq n_0,
    \]
    implying
    \[
        N\rho^*(A,N-1)^d \geq Nr_n^d \geq c(p,d) (M-\epsilon).
    \]
    The desired estimate follows by taking $ M = \underline{\m}_d(A) $ and $ \epsilon =  \underline{\m}_d(A)/2$.
\end{proof}
By the same argument as the above lemma, we have also the following related fact, to be used in the proof of stability in Lemma~\ref{lem:stability}.
\begin{lemma}
    \label{lem:covering_via_packing}
    Suppose $ A\subset \mathbb R^p $ is a compact set with $ 0 < \overline{\m}_d(A) < \infty  $; then for any $ r \leq r_0(A) $ there exists a covering configuration $ \omega_n \subset A $ with radius $ r $, consisting of $ n \leq c(p,d) r^{-d} \overline{\m}_d(A) $ points.
\end{lemma}
\begin{proof}
    Let $ r_0 = r_0(A) $ be such that $ \mathcal L_p(B_r(A)) < 2v_{p-d}r^{p-d}\overline{\m}_d(A) $ for $ r \leq r_0 $.
    The set $ \omega_n\subset A $ is constructed as a subset of the maximal cardinality $ n $, for which whenever $x,y\in \omega_n$ we have $|x-y|\geqslant r$. Then, by maximality of $\omega_n = \{ x_i \}_1^n $, there holds $ A \subset B_r(\omega_n) $, and also, by definition,
    \[
        \mathcal L_p \left(\B{r/2}{x} \cap \B{r/2}{y}\right) = 0, \qquad x\neq y,\ x,y \in \omega_n.
    \]
    Taking into account that 
    \[
        \bigcup_{x\in \omega_n} \B{r/2}{x} \subset \B{r}{A};
    \]
    we conclude for $ N\geq N_0(A,D,\delta) $,
    \begin{equation}
        \begin{aligned}
        n \leq \frac{\L_p[\B{r}{A}]}{v_p (r/2)^p} 
        &\leq
        \frac{2\,v_{p-d} r^{p-d}\,\overline{\m}_d(A) }{v_p (r/2)^p} 
        = 2^{p+1} r^{-d} \frac{v_{p-d}}{v_p} \, \overline{\m}_d(A).
    \end{aligned}
\end{equation}
\end{proof}
We now obtain our first main result for unconstrained covering and polarization.
\begin{proof}[Proof of Theorems~\ref{thm:covering_asmpt} and~\ref{thm:polar_asmpt} for $\rho^*$ and $\P_s^*$.]
    We first establish the result for $d=p$. Recall that restricted to compact sets in $ \mathbb R^d $, $ \m_d = \h_d = \L_d $. Consider the case $ F^- = \rho^* $, so that $ \sigma = -1/d $; fix an $ \epsilon > 0 $ and a compact set $ A\subset \mathbb R^d $. It suffices to assume $ \h_d(A) > 0 $, since otherwise Theorem~\ref{thm:graf_luschgy} applies. For $ \delta = \delta(\epsilon, d) $ as in the stability property, let $ J_\epsilon \supset A $ be a finite union of closed dyadic cubes, such that $  \L_d(A) > (1 - \delta )\L_d(J_\epsilon) $; then $ J_\epsilon $ is Jordan-measurable. Note that $ \delta $ is chosen uniformly in $ J_\epsilon $ and only depends on $ \ep $ and $ d $. Without loss of generality, $ \delta \leq \epsilon < 1 $. On the one hand, monotonicity property together with asymptotics on $ J_\epsilon $ give, with $\theta_d$ as in Theorem \ref{thm:graf_luschgy}:
    \[
        \limsup_{N\to \infty} N^{1/d}\rho^*(A, N) \leq  \limsup_{N\to \infty} N^{1/d}\rho^*(J_\ep, N) = \theta_d \L_d(J_\epsilon)^{1/d} \leq (1-\epsilon)^{-1/d} \theta_d\L_d(A)^{1/d}
    \]
    On the other, by the choice of $ J_\epsilon $ and \eqref{eq:stability} with $ \sgn \sigma = -1 $ and $ \Lim{} = \liminf $,
    \[
       \theta_d \L_d(A)^{1/d} \leq \theta_d \L_d(J_\epsilon)^{1/d} =\liminf_{N\to \infty} N^{1/d}\rho^*(J_\ep, N)\leq  
         (1 + \epsilon) \liminf_{N\to\infty}  N^{1/d}\rho^*(A, N).
    \]
    By making $ \epsilon > 0 $ arbitrary small, we conclude
    \[
            \lim_{N\to \infty} N^{1/d}\rho^*(A, N) = \theta_d \L_d(A)^{1/d}.
    \]
    In the case $ F^+ = \mathcal P_s^* $, $ \sigma = s/d $, the same argument gives
    \[
        \limsup_{N\to\infty} (1 - \epsilon) \cdot \frac{\mathcal P_s^*(A,N)}{N^\sigma} \leq \lim_{N\to\infty}  \frac{\mathcal P_s^*(J_\epsilon,N)}{N^\sigma} \leq \liminf_{N\to\infty}  \frac{\mathcal P_s^*(A,N)}{N^\sigma},
    \]
    which completes the proof when $ A $ is a compact subset of $ \mathbb R^d $.

    Suppose now $ A\subset\mathbb R^p $ is an $ (\h_d,d) $-rectifiable compact set with $ \h_d(A) = \m_d(A) > 0 $ and fix an $ \epsilon > 0 $. Take $ \delta = \delta(\epsilon, s,p,d,A) $ to be from the stability property. By a result of Federer (see \cite[Thm. 3.2.18]{federerGeometric1996} or \cite[Proposition 2.76]{ambrosioFunctions2000}), there exist compact sets $ K_1, \ldots, K_M \subset \mathbb R^d $ and bi-Lipschitz maps $ \psi_m:K_m\to \mathbb R^{p} $, $ 1\leq m \leq M $ with constant $ 1+\epsilon $, such that $ \psi_m(K_m) \subset  A $ are disjoint and
    $$ \mathcal H_d\left( A\setminus \bigcup_{m=1}^M \psi_m(K_m)\right) < \delta, $$
   Denote
    \[
        \wt A = \bigcup_{m=1}^M \psi_m(K_m),
    \]
    a compact $ (\h_d,d) $-rectifiable set, which satisfies $ \h_d(\wt A) = \m_d(\wt A)$ (see \cite[Lemma 4.3]{borodachovLow2014}). Without loss of generality, sets $ A_m := \psi_m(K_m) $ satisfy $ \h_d(A_m) > 0 $. In particular, this implies that $\m_d(A_m)=\h_d(A_m)>0$ for every $m=1,\ldots,M$. We need this estimate to utilize the short-range properties of $\f^\pm$. 

    Let us first consider the covering functional, $ \f^- = \rho^* $, so that $ \sigma =-1/d $. 
		Take a sequence $\mathcal{N}$ that attains the $\liminf_{N\to \infty} \rho^*(\tilde{A}, N)N^{1/d}$. For every $N\in \mathcal{N}$, pick a configuration $\omega_N^*$ attaining $\rho^*(A, N)$. Since the sets $A_1, \ldots, A_m$ are compact and metrically separated, we have
		$$
		h:=4\min_{j\not=k}(\text{dist}(A_j, A_k)) > 0.
		$$
		Denote $N_m:=\#(\omega_N^*\cap \B h {A_m})$. Since $m=1,\ldots,M$ has a finite range, we can pass to a subsequence of $ \mathcal N $, to ensure that the limits
		$$
		\beta^*_m:=\lim_{N\to \infty} \frac{N_m}N
		$$
		exist for every $m=1,\ldots, M$.
        By the short-range property of $\f^-=\rho^*$, we have
   		\[
        \begin{aligned}
            &\liminf_{N\to \infty} \frac{\rho^*(\wt A,N)}{N^{-1/d}} = \lim_{N\in \mathcal{N}}\frac{\rho^*(\wt{A}, N)}{\max_{m} \rho^*(A_m, N_m)} \cdot \frac{\max_{m} \rho^*(A_m, N_m)}{N^{-1/d}}  \\
           \geqslant & \liminf_{N\to \infty} \frac{\max_m \rho^*(A_m, N_m)}{N^{-1/d}} \geqslant \max_m \liminf_{N\to \infty}\frac{\max_m \rho^*(A_m, N_m)}{N_m^{-1/d}} \cdot \frac{N_m^{-1/d}}{N^{-1/d}}\\
        &\geq \max_m \left((\beta_m^*)^{-1/d}\liminf_{N\to \infty}  \frac{ \rho^*(A_m, N_m)}{N_m^{-1/d}}\right).
				\end{aligned}
				\]
Recall that $A_m = \psi_m(K_m)$, where $K_m\subset \R^d$ is a compact set, and $\psi_m$ is a $(1+\ep)$-bi-Lipschitz map, and so
        \[
        \begin{aligned}
            &\liminf_{N\to \infty} \frac{\rho^*(\wt A,N)}{N^{-1/d}}\\
						&\geq \frac1{1+\epsilon} \max_m \left((\beta_m^*)^{-1/d}\liminf_{N\to \infty}  \frac{ \rho^*(K_m, N_m)}{N_m^{-1/d}}\right)\\
        &= \frac1{1+\epsilon} \max_m \left((\beta_m^*)^{-1/d} \theta_{d} \L_d(K_m)^{1/d}\right) \\
        &\geq \frac{\theta_{d}}{(1+\epsilon)^2} \max_m \left((\beta_m^*)^{-1/d}  \h_d(A_m)^{1/d}\right). \\
        \end{aligned}
    \]
		
We finally notice that $N_1+\ldots+N_M \leqslant N$, and so 
$$
\sli_{m=1}^M \beta_m^* \leqslant 1 = \sli_{m=1}^M \frac{\h_d(A_m)}{\h_d(\tilde{A})}.
$$
Thus, for some $m$ we have $\beta_m^* \leqslant \h_d(A_m)/\h_d(\tilde{A})$, and therefore
$$
  \liminf_{N\to \infty} \frac{\rho^*(A,N)}{N^{-1/d}} \geqslant\liminf_{N\to \infty} \frac{\rho^*(\wt A,N)}{N^{-1/d}} \geqslant \frac{\theta_{d}}{(1+\epsilon)^2} \h_d(\tilde{A})^{1/d}.
$$
Since $\ep$ can be made arbitrarily small, we obtain
$$
 \liminf_{N\to \infty} \frac{\rho^*(A,N)}{N^{-1/d}} \geqslant \theta_d \h_d(A)^{1/d}.
$$

\noindent To finish our proof for the covering radius, it is now enough to show that
$$
\limsup_{N\to \infty} \frac{\rho^*(A,N)}{N^{-1/d}}\leqslant \theta_d \h_d(A)^{1/d}.
$$
  Indeed, take the same set $\wt{A}$, and for every $N$ fix numbers $N_1, \ldots, N_M$ with $N_1+\ldots+N_M=N$ and such that
	$$
	\beta_m:=\lim_{N\to \infty} \frac{N_m}N = \frac{\h_d(A_m)}{\h_d(\wt{A})}.
	$$ 
	Then by the short-range property, we have
    \[
        \begin{aligned}
      \limsup_{N\to \infty} \frac{\rho^*(\wt A,N)}{N^{-1/d}} 
                       &\leq\limsup_{N\to \infty} \frac{\rho^*(\wt A,N)}{\max_m \rho^*(A_m, N_m)} \cdot \frac{\max_m \rho^*(A_m, N_m)}{N^{-1/d}}\\
            &= \limsup_{N\to \infty}\frac{\max_m \rho^*(A_m, N_m)}{N^{-1/d}}\\
						&\leq \max_m\left( \limsup_{N\to\infty} \frac{\rho^*(A_m, N_m)}{N^{-1/d}}\right)\\
						&=\max_m \left( \beta_m^{-1/d}\limsup_{N\to\infty} \frac{\rho^*(A_m, N_m)}{N_m^{-1/d}}\right)\\
            &\leq (1+\epsilon)\max_m \left(\beta_m^{-1/d} \limsup_{N\to \infty} \frac{ \rho^*(K_m, N_m)}{N_m^{-1/d}} \right)\\
            &= (1+\epsilon) \max_m \left(\beta_m^{-1/d} \theta_{d} \L_d(K_m)^{1/d}\right) \\
            &\leq {(1+\epsilon)^2}{\theta_{d}} \max_m \left(\beta_m^{-1/d}  \h_d(A_m)^{1/d}\right) \\
						&= (1+\ep)^2 \theta_d \h_d(\wt{A})^{1/d} \\
						&\leq(1+\ep)^2 \theta_d \h_d(A)^{1/d}.
        \end{aligned}
    \]
 It remains to recall that the choice of $\wt{A}$ was defined by the stability property, and therefore
$$
 \limsup_{N\to \infty} \frac{\rho^*(A,N)}{N^{-1/d}} \leqslant  (1+\ep)\limsup_{N\to \infty} \frac{\rho^*(\wt A,N)}{N^{-1/d}} \leqslant (1+\ep)^3 \theta_d \h_d(A)^{1/d}.
$$   
Since $\ep$ can be made arbitrarily small, our proof for the covering radius is finished.

\bigskip

In the case $\sigma=s/d$; i.e., the polarization $\P^*_s$, the proof is the same and we therefore omit it here.
\end{proof}

As a consequence of the asymptotics of optimal covering/polarization, obtained above, we deduce the uniformity of asymptotic distribution for configurations $ \omega_N $ achieving the optimal constant in the limits of $ N^{1/d R(\omega_N,A)} $ and $ P_s(A,\omega_N) / N^{s/d} $, respectively. 

\begin{proof}[Proof of Theorem~\ref{thm:distr}.]
    We shall give the proof of uniform distribution for asymptotically optimal covering configurations; the proof for optimal polarization is similar. Let $ \omega_N $, $ N\geq 1 $, be a sequence of $ N $-point configurations as in the statement of the theorem, and let $ \mu $ be a cluster point of the corresponding empirical probability measures. That is, for some subsequence $\mathcal{N}\subset \mathbb{N}$ we have
		$$
		\frac{1}{N}\sli_{x\in \omega_N} \delta_x \weakto \mu, \;\;\; N\to \infty, \;\;\; N\in \mathcal{N}.
		$$
In what follows, we always assume $N\subset \mathcal{N}$. 

    We begin the proof by showing that $ \mu \ll \h_d $. Indeed, otherwise there exists a closed $ D\subset A $ such that $ \h_d(D) = 0 $, but $ b:= \mu(D) > 0 $. As a result, for any $ \epsilon > 0 $ there are closed neighborhoods $ E_1 := A\cap B_r(D) $ and $ E_2 := A\cap B_{2r}(D) $ with $ r>0 $, such that $ \mu(\partial_A E_1) = 0 $, where $ \partial_A $ denotes the boundary relative to $ A $, and such that $ \h_d(E_2) < \epsilon $. 
		Condition $\h_d(\partial_A E_1)=0$ implies, by definition of weak$^*$ convergence, that  
		\[
        \lim_{N\to \infty} \frac{\#(\omega_N \cap A\setminus E_1)}N =\mu(A\setminus E_1) = 1-\mu(E_1) \leq 1-b.
    \]
		Since for $ N $ large enough we have $ R(\omega_N,A)\leqslant C N^{-1/d} $ is small, Lemma~\ref{lem:covering_via_packing} implies that there is a collection $ \omega_n $ for set $ E_2 $ with covering radius $ R(\omega_N,A) $ and cardinality
    \[
        n \leq c_1(p,d,A) {R(\omega_N,A)}^{-d} \h_d(E_2) \leqslant \epsilon c_2(p,d, A) N .
    \]
       
    We define a new sequence of configurations(of cardinality not necessarily equal to $N$):
    \[
        \omega^{(N)} := (\omega_N \setminus E_1) \cup \omega_n.
    \]
    For $N$ large enough, we have $R(\omega_N, A)<r$, which implies that $\omega^{(N)}$ covers $A$ with radius at most $R(\omega_N, A)$. 
   However, cardinalities of these new configurations are such that 
	$$
	\lim_{N\to\infty}\frac{\#\omega^{(N)}}N \leqslant 1 - b + c(p,d,A)\epsilon < 1
	$$
	for $ \epsilon $ sufficiently small. We have therefore found configurations with cardinalities at most a fraction of $ N $, but with the same covering radius. This contradicts the assumption of asymptotic optimality of $ \omega_N $.

    To prove that the measure $ \mu $ is uniform on $ A $ with respect to $ \h_d $, we again argue by contradiction. Let $ \phi = d\mu/d\h_d $; suppose it is not constant $ \h_d $-a.e. As a result, there exists a pair of distinct points $ x_1, x_2 \in A $, such that 
    \[
        \lim_{r\downarrow 0} \frac{\mu\left(B_r(x_m)\right)}{\h_d\left(B_r(x_m)\cap A\right)} = \phi_m, \qquad m=1,2, \qquad \phi_1 < \phi_2.
    \]
    Then for any $ \epsilon > 0 $ there exists a pair of radii $ r_1, r_2 $, for which $ r_1+r_2 < \|x_1-x_2\| $ and
    \[
        \left|\frac{\mu\left(B_{r_m}(x_m)\right)}{\h_d\left(B_{r_m}(x_m)\cap A\right)} - \phi_m\right| < \epsilon.
    \]
    Let $ \epsilon < (\phi_2-\phi_1)/4 $ and $ B_m :=  B_{r_m}(x_m)$, $ m=1,2 $ be a pair of such closed balls; without loss of generality also $ \h_d\left(\partial_A B_m\right) = 0 $, $ m=1,2 $. Consequently, there is another pair of radii $ t_m < r_m $, such that for  $ \tilde B_m := B_{t_m}(x_m)  $ there holds $ \h_d(\partial_A \tilde B_m) =0 $ and
    \begin{equation}
        \label{eq:ball_in_a_ball}
        \left|\frac{\mu\left(B_m\right)}{\h_d(\tilde B_m\cap A)} - \phi_m\right| < \epsilon 
        \qquad \text{and} \qquad
        \left|\frac{\mu(\tilde B_m)}{\h_d\left(B_m\cap A\right)} - \phi_m\right| < \epsilon,\quad m=1,2.
    \end{equation}
    Since the sets $ \tilde B_m $ and $ A\setminus (B_1\cup B_2) $ are positive distance apart, the short-range property \eqref{eq:stab_1} gives for a sufficiently large $ N_0 $,
    \begin{equation}
        \label{eq:proving_dist}
        \rho(A,N)^d \geq \rho\left( \tilde B_1\cap A,\,\lfloor \mu\left(B_1\right)N \rfloor\right)^d, \qquad N \geq N_0,
    \end{equation}
     as the fraction of points in $ \omega_N $ within distance $ r_m - t_m > 0 $ from  $ \tilde B_m $ is at most $ \mu\left(B_m\right) $, $ m=1,2 $. Applying Theorem~\ref{thm:covering_asmpt} to the set $ \tilde B_1\cap A $ results in
    \[
        \liminf_{N\to \infty} (\mu\left(B_1\right) N)\, \rho\left(\tilde B_1\cap A,\, \lfloor \mu\left(B_1\right)N \rfloor\right)^d \geq \theta_d^d \h_d\left(\tilde B_1\cap A\right),
    \]
    whence, dividing through by $ \mu(B_1) $, one has from \eqref{eq:ball_in_a_ball} and \eqref{eq:proving_dist}
    \[
        \lim_{N\to \infty} N \cdot \rho(A,N)^d \geq \limsup_{N\to \infty} N \rho\left( \tilde B_1\cap A,\,\lfloor\mu\left(B_1\right)N \rfloor\right)^d \geq \theta_d^d/(\phi_1+\epsilon).
    \]

    The last inequality controls how good the covering by a non-uniformly distributed sequence $ \omega_N $ can be. It remains to show that, because there is a large number of points in $ \omega_N \cap \tilde B_2 $, some of them can be removed without making the local covering worse than in the bound we just obtained.
    Indeed, by Theorem~\ref{thm:covering_asmpt}, using a collection $ \omega_n $ of cardinality $ n(N) := \lfloor (\phi_1+2\epsilon)\h_d\left(B_2\cap A\right) N \rfloor $, the set $ B_2 \cap A $ can be covered in an optimal fashion, to achieve the covering radius satisfying
    \[
        \lim_{N\to \infty} n(N) \, \rho(B_2\cap A,\ n(N))^d = \theta_d^d\h_d\left(B_2\cap A\right),
    \]
    whence 
    \[
        \lim_{N\to \infty} N\, \rho(B_2\cap A,\ n(N))^d \leq \theta_d^d/(\phi_1 + 2\epsilon) < \lim_{N\to \infty} N \cdot \rho(A,N)^d.
    \]
    Replacing $ \omega_N \cap \tilde B_2 $ with the optimal covering configuration for $ A\cap B_2 $ results in a configuration of the form
    \[
        \omega^{(N)} = (\omega_N\setminus \tilde B_2) \cup \omega_n.
    \]
    Using the short-range property~\eqref{eq:stab_2}, we see that $ \omega^{(N)} $ has the same covering radius as $ \omega_N $ on $ \omega_N\setminus B_2 $ for large $ N $, and its covering radius on $ B_2 $ has asymptotics smaller than those of $  N^{1/d} \rho(A,N) $, hence $ \omega^{(N)} $ is asymptotically optimal on $ A $.

    In addition, by \eqref{eq:ball_in_a_ball} and the definition of $ n(N) $, the fraction $ n(N)/N $ of points required for covering of $ B_2 $ in $ \omega^{(N)} $ is strictly smaller than the asymptotic fraction $ \mu(\tilde B_2) $ of points from $ \omega_N $ contained in $ \tilde B_2 $. Thus,
    the difference of cardinalities of $ \omega^{(N)} $ and $ \omega_N $ is a positive fraction of $ N $:
    \[
        N \cdot \left(\mu(\tilde B_2) - \frac{n(N)}{N} \right) > 0,
    \]
    which contradicts the asymptotic optimality of $ \omega_N $ assumed in the theorem. We conclude that $ \mu $ must be uniform with respect to $ \h_d $ on $ A $.
\end{proof}

Before we turn to the proof of the following theorem, we shall introduce a simple observation about the set neighborhoods $ \B r A $. Namely, given two compact sets $ E,F\subset \mathbb R^p $ and positive numbers $ \alpha, r > 0 $, there holds
\begin{equation}
    \label{eq:inclusion}
    \B {r} E\setminus \B {(1+\alpha) r} {F} \subset \B {r} {E\setminus \B {\alpha r} {F}}.
\end{equation}
Indeed, pick an element $ x \in \B {r} E\setminus \B {(1+\alpha) r} {F} $; then for some $ y\in E $, $ \|x-y\|\leq r $. Furthermore, since $ \dist (x,{F}) > (1+\alpha)r $, $ \dist (y,{F}) > \alpha r $. It follows $ y\in E\setminus \B {\alpha r} F $, and therefore $ x \in \B {r} {E\setminus \B {\alpha r} {F}} $ as desired.

\begin{proof}[Proof of Theorem~\ref{thm:minkisbigger}]
    Clearly, to prove the first pair of inequalities in the statement of the theorem, it suffices to show
    \[
        \limsup_{N\to \infty}N^{1/d} \rho^*(A, N) > \theta_d \H_d(A)^{1/d}.
    \]
    Furthermore, the second claim of the theorem follows from Corollary~\ref{cor:limiting_polar}. We will therefore focus on proving the above inequality for $ \rho^* $.

    Fix an $ \epsilon > 0 $.
    As in the preceding proof of this section, an application of \cite[Thm. 3.2.18]{federerGeometric1996} gives existence of a compact set $ \wt A\subset A $, such that  $ \wt A $ is $ d $-rectifiable and $ \m_d(\wt A) = \h_d(\wt A) \geq \h_d(A) -\epsilon $. Consequently, Theorem~\ref{thm:covering_asmpt} applies to $ \wt A $; by the theorem we have
    \begin{equation}
        \label{eq:covering_tilde}
            \lim_{N\to \infty} N \rho^*(\wt A,N)^d = \theta_d^d \h_d(\wt A) \geq\theta_d^d (\h_d(A) - \epsilon).
    \end{equation}
    We shall further need to characterize $ A\setminus \wt A $. By the definition of Minkowski content, there is a sequence of positive numbers $ r_n\downarrow0 $, $ n\geq 1 $, such that along this sequence, 
    \[
        \lim_{n\to \infty}\frac{\mathcal L_p[\B {r_n} A]}{v_{p-d}r_n^{p-d}}= \overline{\m_d}(A), \qquad \lim_{n\to \infty}\frac{\mathcal L_p[\B {r_n} {\wt A}]}{v_{p-d}{r_n}^{p-d}} = \h_d(\wt A). 
    \]
    In particular, there exists $ n_0(\epsilon, A,\wt A) $ so large that
    \[
        \frac{\mathcal L_p[\B {r_n} A]}{v_{p-d}{r_n}^{p-d}}\geq \overline{\m_d}(A)-\epsilon, \qquad \left|\frac{\mathcal L_p[\B {r_n} {\wt A}]}{v_{p-d}{r_n}^{p-d}} - \h_d(\wt A)\right| \leq \epsilon, \quad n \geq n_0,
    \]
    implying for every $ \alpha > 0 $ sufficiently small
    \begin{equation}
        \label{eq:lower_minkowski_difference}
        \begin{aligned}
            \mathcal L_p[\B {r_n} A\setminus \B {(1+\alpha) r_n} {\wt A}] 
            &\geq \mathcal L_p[\B {r_n} A] - \mathcal L_p[\B {(1+\alpha) r_n} {\wt A}] \\
            &\geq v_{p-d} {r_n}^{p-d} (\overline{\mathcal M_d}(A) - (1+\alpha)^{p-d} \mathcal H_d(\wt A) -2 \epsilon)\\
            &\geq v_{p-d} {r_n}^{p-d} (\overline{\mathcal M_d}(A) - (1+\alpha)^{p-d} \mathcal H_d(A) - 3 \epsilon).
    \end{aligned}
    \end{equation}
    To apply the above estimate, consider a covering set $ \wt \omega $ for $ A\setminus \B {\alpha r_n} {\wt A} $ of cardinality $ k_n $, achieving the covering radius of at most $ r_n $, $ n\geq 1 $. Using \eqref{eq:inclusion} and that by the covering property $ A\setminus \B {\alpha r_n} {\wt A} \subset \B {r_n} {\wt \omega} $ gives
    \[
        \B {r_n} A\setminus \B {(1+\alpha) r_n} {\wt A} \subset \B {r_n} {A\setminus \B {\alpha r_n} {\wt A}} \subset \B {2r_n} {\wt \omega}.
    \]
    In view of~\eqref{eq:lower_minkowski_difference}, this results in the following estimate for $ k_n = \#\wt \omega $:
    \begin{equation}
        \label{eq:boundI}
        k_n  r_n^d \geq \frac{v_{p-d}}{2^pv_p} (\overline{\mathcal M_d}(A) - (1+\alpha)^{p-d} \mathcal H_d(A) - 3 \epsilon) > 0
    \end{equation}
    for sufficiently small $ \alpha $, $ \epsilon $.

    Now consider an optimal covering configuration $ \omega_N $ for $ A $ attaining the covering radius $ \rho:= R(\omega_N, A) < \alpha r_n/3 $ and such that $ \rho^*(A, N-1) \geq \alpha r_n/3 $,
    so that $ N $ is the smallest possible cardinality for this inequality. Denote 
    \[
        \omega' := \omega_N \cap \B \rho {\wt A}, \qquad \omega'' := \omega_N \setminus \B {2\rho} {\wt A}.
    \]
    By construction, since $ \alpha r_n > 3 \rho $
    \[
        R(\omega', \wt A) = \rho = R(\omega'', A\setminus \B {\alpha r_n}{\wt A}).
    \]
    Since $ N_1\to \infty $ when $ r_n\downarrow 0 $, by taking $ n $ large enough and using~\eqref{eq:covering_tilde} we ensure that 
    \[
        N_1 (r_n/3)^d \geq N_1 \rho^*(\wt A, N_1)^d \geq \theta_d^d(\h_d(A) - 2\epsilon);
    \]
    in addition, due to $ \omega_N $ having the covering radius $ \rho < r_n $ and~\eqref{eq:boundI}, there holds
    \[
        N_2  (\alpha r_n/3)^d \geq \left(\frac{\alpha}3\right)^d \cdot \frac{v_{p-d}}{2^pv_p} (\overline{\mathcal M_d}(A) - (1+\alpha)^{p-d} \mathcal H_d(A) - 3 \epsilon) > 0.
    \]
    Adding together the last two displays yields, by fixing sufficiently small $ \alpha, \epsilon >0 $ and taking $ n\to \infty $:
    \[
        \limsup_{N\to \infty} N\rho^*(A,N) \geq \limsup_{n\to \infty} 
        (N(n)-1)\rho^*(A,N(n)-1) \geq \limsup_{n\to \infty} (N_1+N_2-1) (\alpha r_n/3)^d > \theta_d^d \h_d(A).
    \]
\end{proof}
\begin{proof}[Proof of Theorem~\ref{thm:generalest}]
    The proof follows along the same lines as that of Lemma~\ref{lem:small_covering}. Indeed, taking $ M = \underline{\m_d}(A) $ the first part of that proof gives
    \[
        \liminf_{N\to \infty}\rho^*(A,N) N^{1/d}  \leq \liminf_{n\to \infty}\rho^*(A,N(n)) \,N(n)^{1/d} \leq c(p,d) \underline{\m_d}(A)^{1/d}.
    \]
    To obtain the inequality
    \[
        \limsup_{N\to \infty}N^{1/d} \rho(A, N)\leqslant c_4 \overline{\m}_d(A) ^{1/d},
    \]
    observe that for any $ \epsilon > 0 $ and sequence $ r_n\downarrow 0 $ one has eventually
    \[
        \L_p(\B{r_n}{A}) \leq (\overline{\m}_d(A) + \epsilon) v_{p-d}{r_n}^{p-d},
    \]
    and it suffices to set $ r_n = \rho^*(A,n) $ over the sequence achieving the $ \limsup_{N\to \infty}N^{1/d} \rho(A, N) $ in the proof of Lemma~\ref{lem:small_covering}. This gives the two upper bounds of the theorem.

    For the lower bounds, start by observing that for every $ \epsilon >0 $ and sequence $ r_n\downarrow 0 $, eventually
    \[
        \L_p(\B{{r_n}}{A}) \geq (\underline{\m_d}(A)-\epsilon)  v_{p-d}{r_n}^{p-d},
    \]
    so it suffices to take $ r_n = \rho^*(A,n) $ over the sequence achieving the $ \liminf_{N\to \infty}N^{1/d} \rho(A, N) $ in the proof of Lemma~\ref{lem:small_covering}. Finally, for the second lower bound of the theorem, take $ M = \overline{\m_d}(A) $ in the third part of Lemma~\ref{lem:small_covering} and observe
    \[
        \limsup_{N\to \infty}\rho^*(A,N) N^{1/d}  \geq \limsup_{n\to \infty}\rho^*(A,N(n)) \,N(n)^{1/d} \geq c(p,d) \overline{\m_d}(A)^{1/d},
    \]
    with the dependence $ N(n) $ as in the lemma.
\end{proof}

\subsection{Verification of the short-range properties of \texorpdfstring{$ \rho^*(A,N) $}{rho*(A,N)} and \texorpdfstring{$ \P_s^*(A,N) $}{Ps(A,N)}}
\label{sec:verification}
It remains to verify properties of the functionals $ R(\omega_N,A) $ and $ P_s(\omega_N, A) $, formulated at the beginning of the previous section.

\medskip

\noindent\ul{Monotonicity}. Equation~\eqref{eq:monotonicity} follows immediately from the definitions of $ \rho^*(A,N) $ and $ \P_s^*(A,N) $.

\medskip

\noindent\ul{Existence of asymptotics on Jordan-measurable sets}. As mentioned previously, existence of these asymptotics for covering was obtained in the papers of Kolmogorov and Tikhomirov \cite{kolmogorovEentropy1959}, and independently the monograph of Graf and Luschgy \cite{grafFoundations2000}. The corresponding results for polarization were proved in \cite{BorodTAMS} for constrained polarization and, using similar ideas, in \cite{hardinUnconstrained2020} for its unconstrained analog.

\medskip

\noindent\ul{Short-range property}. We first prove the estimate \eqref{eq:stab_1} for the covering functional $\f^- = \rho^*$. By Lemma \ref{lem:small_covering}, for some constant $C= C(s,p,d,A)>0$ there holds, in view of the assumption $\overline{\m}_d(A_{m})<\infty$, $ m=1,2 $:
$$
\rho^*(A_{1}, N)\leqslant CN^{-1/d}, \;\;\;\; \rho^*(A_{2}, N)\leqslant CN^{-1/d}, \;\;\;\; \rho^*(A_{1}\cup A_2, N)\leqslant CN^{-1/d}.
$$
For the sake of brevity, denote $A:=A_1\cup A_2$, and fix a configuration $\omega_N$ such that $\rho^*(A, N)=R(\omega_N,A)$. For $m=1,2$, we have
\begin{multline}
CN^{-1/d}\geqslant \rho(A, N)=R(\omega_N, A_1\cap A_2)\geqslant R(\omega_N, A_m)\geqslant\min_{x\in \omega_N} |y-x| \\= \max_{y\in A_m} \min\Big(\min_{x\in \omega_N \cap \B h {A_m}} |y-x|, \min_{x\in \omega_N^* \setminus \B h {A_m}} |y-x|\Big) \geqslant \min\Big(\min_{x\in \omega_N^* \cap \B h {A_m}} |y-x|,h\Big).
\end{multline}
Since $h$ is fixed, and for large $N$ we have $h>CN^{-1/d}$, we obtain that the latter minimum is equal to $\min_{x\in \omega_N \cap \B h {A_m}} |y-x|$, and therefore, by definition of $N_m$,
$$
\rho(A_m, N_m) \leqslant R(\omega_N\cap \B h {A_m}, A_m) \leqslant R(\omega_N, A_m)\leqslant R(\omega_N, A)=\rho^*(A, N).
$$
This implies
$$
\max\left(\rho^*(A_1, N_1), \rho^*(A_2, N_2)\right) \leqslant \rho^*(A_1\cup A_2, N),
$$
and the desired inequality is proved. 

\hfill $\square$

We proceed with proving \eqref{eq:stab_1} for the polarization functional $\f^+=\P_s$. In this case, Lemma \ref{lem:small_covering} implies
\begin{equation}\label{eq:aux_stab_1}
\P_s(A_{1}, N)\geqslant CN^{s/d}, \;\;\;\; \P_s(A_{2}, N)\geqslant CN^{s/d}, \;\;\;\; \P_s(A_{1}\cup A_2, N)\geqslant CN^{s/d}.
\end{equation}
Again denote $A:=A_1\cup A_2$ and notice that, for $y\in A_m$, $m=1,2$, we have
$$
\sli_{x\in \omega_N\setminus \B h {A_m}} |y-x|^{-s} \leqslant h^{-s} \cdot N
$$
and, therefore,
$$
\sli_{x\in \omega_N} |y-x|^{-s} \leqslant \sli_{x\in \omega_N \cap \B h {A_m}}|y-x|^{-s} + h^{-s}\cdot N.
$$
This implies
$$
\begin{aligned}
\P_s(A, N) = P_s(\omega_N, A) &= \inf_{y\in A} \sli_{x\in \omega_N} |y-x|^{-s} = \min\left(\inf_{y\in A_1}\sli_{x\in \omega_N} |y-x|^{-s}, \inf_{y\in A_2}\sli_{x\in \omega_N} |y-x|^{-s}\right) \\ 
&\leqslant \min\left(\inf_{y\in A_1}\sli_{x\in \omega_N\cap A_1(h)} |y-x|^{-s}, \inf_{y\in A_2}\sli_{x\in \omega_N\cap A_2(h)} |y-x|^{-s}\right) + h^{-s}\cdot N\\
&\leqslant \min\left(\P_s(A_1, N_1), \P_s(A_2, N_2)\right) + h^{-s} \cdot N.
\end{aligned}
$$
It now remains to divide the above inequality by $\P_s(A, N)$ and pass to the limit, taking into account estimates \eqref{eq:aux_stab_1} and that $s>d$. 

\hfill $\square$

We now prove \eqref{eq:stab_2} for $\f^-=\rho^*$. Take two configurations, $\omega_{N_m}$, $m=1,2$ such that $\rho^*(A_m, N_m)=R(\omega_{N_m}, A_m)$. Define $\omega_N:=\omega_{N_1}\cup \omega_{N_2}$; then $\#\omega_N = N_1+N_2 =: N$. Therefore,
\begin{multline*}
\rho^*(A, N) \leqslant R(\omega_N, A) = \max\left(R(\omega_N, A_1), R(\omega_N, A_2)\right) \\ \leqslant \max\left(R(\omega_{N_1}, A_1), R(\omega_{N_2}, A_2)\right) = \max\left(\rho^*(A_1, N_1), \rho^*(A_2, N_2)\right),
\end{multline*}
and \eqref{eq:stab_2} is established. The proof of this inequality for $\f^+=\P_s$ is the same, and we leave it to the curious reader. 

\hfill $\square$.
\medskip

\noindent\ul{Stability}.
We prove \eqref{eq:stability} for $\f^+=\P_s^*$. The proof for $\f^-=\rho^*$ will be a by-product of Lemma \ref{lem:stability}, see Corollary~\ref{cor:stability_covering}.

To begin, we prove that in a sufficiently small neighborhood $ \B{r}A $ of a set $ A $, with $ r $ depending on $ N $, the value of $P_s(\omega_N, \B{r}A) $ is close to $ P_s(\omega_N, A) $ for any configuration $\omega_N\subset \R^p$. To do this, we will need the following well-known application of the Frostman lemma.
\begin{lemma}[Theorem 2.4, \cite{erdelyiRiesz2013}] For any Borel set $ A\subset \mathbb R^p $ with $ \h_d(A)>0 $ and $ s>d $, we have
\begin{equation}
    \label{eq:from_frostman}
    \mathcal P_s(\omega_N, A) \leq N^{s/d} \, {\cf(s,d,A)} = N^{s/d}\frac{s}{s-d}\cdot \frac{(2c)^{s/d}}{\mu(A)^{s/d}},\qquad N \geq 1,
\end{equation}
where $\omega_N$ is an $N$-point configuration in $\R^p$, and $ \mu $ is a Borel measure satisfying $ \mu(A) > 0 $ and $ \mu(\B{r}{x}\cap A) \leq cr^d $. Clearly, $ \mu \ll \h_d $.  When $ A\subset \mathbb R^d $, one can use $ \mu = \mathcal H_d $, so that 
\[
    \cf = \frac{s}{s-d}\cdot \frac{(2v_d)^{s/d}}{\mathcal H_d(A)^{s/d}}.
\] 
\end{lemma}
We proceed with the following statement, in which we assume $s>d$ as usual.
\begin{lemma}
    \label{prop:nbhd}
    Suppose the compact set $ A \subset \mathbb R^p $ is such that $ \mathcal H_d(A) > 0 $. Take any $ 0< \epsilon < {1}/{2\cf^{1/s}} $ and define $ r_N := \epsilon N^{-1/d} $.
    Then for any configuration $ \omega_N \subset \R^p $, the sets $ B_{r_N}(A) $ satisfy
    \[
        \begin{aligned}
            \mathcal P_s^*(\omega_N, B_r(A)) 
        &\geq \left(1-\epsilon \, s\, {\cf^{1/s}}\right) \mathcal P_s^*(\omega_N, A)\\ 
        & =:(1 - \epsilon\, c_0(s,d,A))\mathcal P_s^*(\omega_N, A).
    \end{aligned}
    \]
\end{lemma}
\begin{proof}
     Let in this proof
    \[
        H = {1}/{\cf^{1/s}}.
    \]
    Our goal is to prove that for every $y\in \B{r_N}{A}$, we have
		$$
		U(y, \omega_N)=\sli_{x\in \omega_N}\|y-x\|^{-s} \geqslant P_s(\omega_N, A).
		$$
		Since $y\in \B{r_N}{A}$, there is an $x\in A$ such that $|y-x|\leqslant r_N = \ep N^{-1/d}$. There are two cases for the location of $x$:
    \begin{enumerate}[(a)]
        \item $ \min_i \{ \|x-x_i\| < (H - \epsilon) N^{-1/d} \} $
        \item $ \min_i \{ \|x-x_i\| \geq (H - \epsilon) N^{-1/d} \} $,
    \end{enumerate}
    which we consider separately. If the first of the above cases holds, we can find an $ x_i' \in \omega_N $ with $\|x-x_i'\| < (H-\ep)N^{-1/d}$.  Then, from \eqref{eq:from_frostman}, we obtain:
    \[
        U(y,\omega_N) \geq \|y - x_i'\|^{-s} \geq (\|y - x\| + \|x - x_i'\|)^{-s} \geq \left(H  N^{-1/d}\right)^{-s} = {\cf N^{s/d}} \geq P_s(\omega_N,A).
    \]
    If the second case holds, we have for every $ 1 \leq i \leq N $:
    \[
        \begin{aligned}
        \frac{ \|y-x_i\|^{-s} }{ \|x-x_i\|^{-s}} 
        &= \left(\frac{ \|x-x_i\|}{ \|y-x_i\| }\right)^{s} 
        \geq  \left(\frac{ \|x-x_i\|}{\|x-x_i\| + \epsilon N^{-1/d}}\right)^s \\
        & =\left( 1- \frac{\epsilon N^{-1/d}}{\|x-x_i\| + \epsilon N^{-1/d}} \right)^s
        \geq \left(1 - \epsilon/H \right)^s \geq 1- \epsilon\, s /H,
    \end{aligned}
    \]
    whence
    \[
        U(y, \omega_N)\geqslant (1- \epsilon\, s /H)U(x,\omega_N) \geqslant (1- \epsilon\, s /H) \inf_{x\in A} U(x, \omega_N) = (1- \epsilon\, s /H)P_s(\omega_N, A).
    \]
    \end{proof}

We now prove stability for the functional $\f^+=\P^*_s$. 
\begin{lemma}
    \label{lem:stability}
    Suppose $ A\subset \mathbb R^p $ is a compact set with $ 0 < \m_d(A) < \infty  $ and $ \h_d(A) > 0 $, and that $ s>d $.
    Given an $ \varepsilon \in (0,1) $, there exists $ \delta = \delta(\varepsilon, s,p,d,A) $ such that whenever a closed $ D\subset A $ satisfies
    \[
        \m_d(D) > ( 1-\delta )\m_d(A) \qquad \text{and} \qquad \h_d(D) > ( 1-\delta )\h_d(A),
    \]
    there holds 
    \begin{equation}
        \label{eq:regularity}
        \Lim{N\to \infty} \frac{\P_s^*(\omega_N^*(A))}{N^{s/d}}  \geq (1-\varepsilon)
        \Lim{N\to \infty}
        \frac{\P_s^*(\omega_N^*(D))}{N^{s/d}}, \qquad \Lim{} \in \{ \liminf,  \limsup\}.
    \end{equation}
    In addition, for $ d=p $, $ \delta $ can be made independent of $ A $.
\end{lemma}
\begin{proof}
    For this argument, let the constant $ 0< \epsilon_0 < 1/2 $ be such that
    \[
        \frac{1-\epsilon_0}{(1+\epsilon_0)^{s/d}}
          \geq 1-\varepsilon.
    \]
    It will be shown that taking
    \[
        \delta \leq C(s,p,d,A)\, {\epsilon_0^{d+1}}
    \]
    yields the desired estimate; in addition, the constant $ C $ in this inequality will be made independent of $ A $ when $ d=p $.
    Denote 
		$$
		r_N:=\frac{\epsilon_0 N^{-1/d}}{2{c_0(s,d,D)}},
		$$ 
		where $ c_0 $ is defined in the preceding lemma. For every $ N \geq 1 $, we write $A$ as a disjoint union
    \[
        A = D_N\cup R_N
    \]
    where
    \[
        D_N :=  B_{2r_N}(D),\qquad R_N:= A\setminus D_N,
    \]
     Let $\omega_N^*(D)$ be a configuration that attains $\P_s^*(D, N)$. This choice of $ r_N $ implies
    \begin{equation}
        \label{eq:from_lemma}
        P_s(\omega_{N}^*(D),\, D_N) \geq (1-\epsilon_0) \P_s^*(D,N).
    \end{equation}

    To construct a sequence of configurations giving the estimate \eqref{eq:stability}, we choose a number $ n = n(\epsilon_0, N) $ and a configuration $\omega_n\subset R_N$, and denote $\omega_{N+n}:=\omega_N^*(D)\cup \omega_n$. In view of the definition of $ \P_s^* $, we have
    \begin{equation}
        \label{eq:minimum_large}
        \P_s^*(A, N+n) \geq \min \left[P_s(\omega_{N}^*(D), D_N), \ \P_s^*(R_N, n)  \right].
    \end{equation}
    The desired estimate will follow if we choose a configuration $\omega_n $ of cardinality $ n $, so that
    \begin{enumerate}[(a)]
        \item\label{it:one} $\P_s^*(R_N,n) \geq P(\omega_n,R_N) \geq \P_s^*(D,N) $,
        \item\label{it:two} $ n = \lfloor \epsilon_0 N \rfloor. $
    \end{enumerate}
    Indeed, assume conditions \eqref{it:one}--\eqref{it:two} have been verified; using \eqref{eq:from_lemma} and \eqref{eq:minimum_large}, in the case $ \Lim{} = \limsup $ the left-hand side of \eqref{eq:stability} can be estimated as
    \[
        \begin{aligned}
            \limsup_{N\to\infty} \frac{\P_s^*(A, N)}{N^{s/d}} &\geq \limsup_{N\to\infty} \frac{\P_s^*(A, N+n)}{(N+n)^{s/d}}
            \geq  \limsup_{N\to \infty} \frac{1}{(N+n)^{s/d}}  \min \left[P_s(\omega_{N}^*(D), D_N), \ \P_s^*(R_N, n)  \right]\\
            &
            \geq (1-\epsilon_0)\limsup_{N\to \infty} \frac{\P_s^*(D,N)}{(N+n)^{s/d}}
            \geq (1-\epsilon)\limsup_{N\to \infty} \frac{\P_s^*(D,N)}{N^{s/d}}.
        \end{aligned}
    \]
    Similarly, when $ \Lim{} = \liminf $, choosing a subsequence $ \mathcal N \subset \mathbb N $ such that 
    \[
        \lim_{\mathcal N \ni N\to\infty} \frac{\P_s^*(A, N)}{N^{s/d}} = \liminf_{N\to\infty} \frac{\P_s^*(A, N)}{N^{s/d}},
    \]
    and taking $ N+n\in \mathcal N $ gives
    \[
        \begin{aligned}
            \liminf_{N\to\infty} \frac{\P_s^*(A, N)}{N^{s/d}} 
            &= \lim_{\mathcal N\ni(N+n)\to\infty} \frac{\P_s^*(A, N+n)}{(N+n)^{s/d}}\\
            &\geq  \liminf_{\mathcal N\ni(N+n)\to\infty} \frac{1}{(N+n)^{s/d}}  \min \left[P_s(\omega_{N}^*(D), D_N), \ \P_s^*(R_N, n)  \right]\\
            &
            \geq (1-\epsilon_0) \liminf_{\mathcal N\ni(N+n)\to\infty} \frac{\P_s^*(D,N)}{(N+n)^{s/d}}
            \geq (1-\epsilon)\liminf_{N\to \infty} \frac{\P_s^*(D,N)}{N^{s/d}}.
        \end{aligned}
    \]

    Now let us verify conditions~\eqref{it:one}--\eqref{it:two}.
    To prove condition \eqref{it:one}, we use that equation \eqref{eq:from_frostman} gives
    \[
        \P_s^*(D, N ) \leq \cf(s,d,D) N^{s/d} = \left(\frac{c_0}s\right)^s N^{s/d}.     
    \]
    Assume that we can choose a set $\omega_n = \{x_1, \ldots, x_n\}$ such that
    $$
    R_N\subset \bigcup_{i=1}^n B_{{s}/{c_0N^{1/d}}}\left(x_i\right).
    $$
    Then for every $y\in R_N$ there is an $x\in \omega_n$ such that $|y-x|\leqslant s/c_0 N^{-1/d}$, which implies
    $$
    \P_s^*(R_N, n) \geqslant P_s(\omega_n, R_N) = \inf_{y\in R_N} \sli_{i=1}^n \|y-x_i\|^{-s} \geqslant \left(\frac{s}{c_0}\right)^{-s} N^{s/d} \geqslant \P_s^*(D, N),
    $$
    so that for such a choice of $\omega_n$, condition \eqref{it:one} is satisfied. It remains to prove that $n$ can be taken to satisfy \eqref{it:two}. 

    Notice that by definition, for any $ r > 0 $,
    \[
        \B {r} {R_N} = B_{r} \big(A\setminus \B{2r}D\big) \subset \B{r}{A} \setminus \B{r}D.
    \]
    so in view of $ \m_d(D) > (1-\delta) \m_d(A) $,
    \begin{equation*}
        \label{eq:minkowski_difference}
        \begin{aligned}
            \overline{\m}_d(R_N) 
            &= \limsup_{r\downarrow0} \frac{\L_p[\B{r}{R_N}]}{v_{p-d}r^{p-d}} \leq \limsup_{r\downarrow0} \frac {\L_p(\B{r}{A}\setminus \B{r}D)}{v_{p-d}r^{p-d}} 
            =  \limsup_{r\downarrow0} \frac{\L_p(\B{r}{A}) - \L_p( \B{r}D)}{v_{p-d}r^{p-d}}  \\
            &= \delta\, \m_d(A).
        \end{aligned}
    \end{equation*}
    By Lemma~\ref{lem:covering_via_packing}, for $ N\geq N_0(A,D,\delta) $ (so that $ r_N $ is sufficiently small) there exists a configuration $ \omega_n \subset R_N $, for which
    $$
        R_N\subset \bigcup_{i=1}^n B_{r_N}\left(x_i\right),
    $$
    with cardinality $ n $ satisfying
    \begin{equation}
        \label{eq:n_estimate}
        \begin{aligned}
        n \leq \, 
        & \delta\, c(p,d) \, r_N^{-d}  \, \m_d(A)\\
        &= \frac{\delta}{\epsilon_0^d} c(p,d) \, s^d \cdot2^{d} \cf(s,d,D)^{d/s} \, {\m_d(A)}\, N.
    \end{aligned}
    \end{equation}
    Since $ \mu \ll \h_d $, there exists a $ \delta_0 = \delta(A) $, such that $ \h_d(A\setminus D) < \delta_0\h_d(A) $ implies $ \mu(D) \geq \mu(A)/2 $, whence by~\eqref{eq:from_frostman} $ \cf(s,d,D) \leq 2^{s/d}\cf(s,d,A) $.
    Substituting the value of $ \cf $ from~\eqref{eq:from_frostman} shows that
    setting
    \[
        \begin{aligned}
        \delta 
        &= \min \left\{ \delta_0,\  \epsilon_0^{d+1} \cdot \frac{\cf(s,d,A)^{-d/s}}{s^d\, 2^{d} c(p,d)} \cdot {\m_d(A)} \right\} \\
        &=: C(s,p,d,A) \,\epsilon_0^{d+1}
    \end{aligned}
    \]
    gives
    \[
        n \leq \epsilon_0 N, \qquad N \geq N_0(A,D,\delta),
    \]
    The above estimate implies $n\leqslant \lfloor \epsilon_0 N \rfloor$. By adding more points to $\omega_n$, we can always ensure that $n = \lfloor \epsilon_0 N \rfloor$, and the proof of \eqref{it:one}--\eqref{it:two} is finished.
		

   To verify the second claim of the lemma, recall that for $ p=d $, $ \mu = \h_d $; substituting this into $ \cf $ gives $ \delta_0 = 1/2 $, and in view of $ \h_d(A) = \m_d(A) $ equation~\eqref{eq:n_estimate} becomes
   \[
       \begin{aligned}
           n \leq \frac{\delta}{\epsilon_0^d} \, 3s^d\cdot2^{2d+2} \left(\frac{s}{s-d}\right)^{d/s} \, N = c(s,d) \,\frac{\delta}{\epsilon_0^d}\,  N.
   \end{aligned}
   \]
   As a result, $ \delta $ is chosen independently of $ A $ in this case, as desired.
\end{proof}
As a by-product of our proof, we get the stability result for $\f^-=\rho^*$ and for the unconstrained polarization.

\begin{corollary}
    \label{cor:stability_covering}
    Since for the proof of \eqref{eq:stability} we presented a covering of the set $ A\setminus \B{2r_N}D $, it is easy to see that this construction yields also the stability results for covering.
\end{corollary}
\begin{corollary}
Assume $A\subset \R^p$ is an $(\h_d, d)$-rectifiable set with $0 < \m_d(A)=\h_d(A) < \infty$, and $s>d$. For every $ \epsilon >0 $, there exists $ \delta = \delta(\epsilon, s,p,d,A) $ such that 
    \begin{equation}
        \label{eq:stability_constr}
        \liminf_{N\to\infty} \frac{\mathcal \P_s(A,N)}{N^{s/d}} \geq  
         (1 - \epsilon) \cdot \liminf_{N\to\infty} \frac{\mathcal \P_s(D,N)}{N^{s/d}}
    \end{equation} 
		and
	 \begin{equation}
        \label{eq:stability_constr_2}
        \limsup_{N\to\infty} \rho(A, N)N^{1/d} \leq  
         (1 + \epsilon) \cdot \limsup_{N\to\infty} \rho(A, N)N^{1/d}.
    \end{equation} 	
    whenever the compact $ D \subset A $ satisfies $ \h_d(D) > (1 - \delta )\,\h_d(A)$. 
\end{corollary}
The proof of this corollary is identical to the proof of \eqref{eq:stability}; indeed, the only configuration we constructed was $\omega_n \subset R_N\subset A$. Therefore, if the original configuration $\omega^*_N(D)$ is a subset of $D$, then the new configuration $\omega_{N+n}$ is a subset of $A$, and we can estimate $\P_s(A, N+n)\geqslant P_s(\omega_{N+n}, A)$. 
\begin{proof}[Proof of Theorems \ref{thm:covering_asmpt} and~\ref{thm:polar_asmpt} for $\rho$ and $\P_s$]
Assume first $p=d$; i.e., $A\subset \R^d$. In this case, our theorem follows from \cite[Theorem 1.11]{hardinUnconstrained2020}. 

Assume now $d<p$. We notice that the only difference of $\rho$ and $\P_s$ from $\rho^*$ and $\P_s^*$ is the lack of monotonicity. However, the proof of the inequalities
$$
\liminf_{N\to \infty} \frac{\mathcal \P^*_s(A,N)}{N^{s/d}} \geqslant \frac{\sigma_{s,d}}{\h_d(A)^{s/d}}, \;\;\;\; \text{and} \;\;\;\; \limsup_{N\to \infty}\rho^*(A, N)N^{1/d} \leqslant \theta_d \h_d(A)^{1/d}
$$
did not use monotonicity and used only stability and short-range properties. Since these properties still hold for $\rho$ and $\P_s$, we get
$$
\liminf_{N\to \infty} \frac{\mathcal \P_s(A,N)}{N^{s/d}} \geqslant \frac{\sigma_{s,d}}{\h_d(A)^{s/d}}, \;\;\;\; \text{and} \;\;\;\; \limsup_{N\to \infty}\rho(A, N)N^{1/d} \leqslant \theta_d \h_d(A)^{1/d}.
$$

On the other hand, we clearly have $\P_s(A, N)\leqslant \P_s^*(A, N)$, and therefore
$$
\limsup_{N\to \infty} \frac{\mathcal \P_s(A,N)}{N^{s/d}} \leqslant \limsup_{N\to \infty} \frac{\mathcal \P^*_s(A,N)}{N^{s/d}}\leqslant \frac{\sigma_{s,d}}{\h_d(A)^{s/d}},
$$
which finishes the proof for $\P_s$. Finally, we have $\rho(A, N)\geqslant \rho^*(A, N)$, which implies
$$
\liminf_{N\to \infty} \rho(A, N)N^{1/d} \geqslant \liminf_{N\to \infty} \rho^*(A, N)N^{1/d}\geqslant \theta_d \h_d(A)^{1/d},
$$
and our proof is finished.
\end{proof}
\begin{remark}
    An inspection of the proof of the above lemma shows that when the set $ A $ is $ d $-regular, i.e.\ equation~\eqref{eq:dregular} holds, the constant $ \delta $ depends only on the constant $ c $ in this equation. Furthermore, one can use $ \mu = \h_d $ in this case.
\end{remark}
\subsection{Covering as a limit of polarization}
Recall the following result about separation properties of optimal polarization configurations. It will be useful to control the cardinality of such configurations in a given volume.
\begin{customthm}{B}[Theorem 2.3 \cite{reznikovCovering2018}, Proposition 4.2 \cite{hardinUnconstrained2020}]
    \label{thm:weak_sep}
    Let $A\subset \R^p$ be a set with $\H_d(A)>0$. For every $s>\max(d, p-2)$ there exists a constant $\eta(s,d,A)$ such that for every $N\geqslant 1$, and every configuration $\omega_N$ that attains $\P_s^*(A, N)$, we have
    $$
    \#\Big(\omega_N \cap \B{\eta(s,d,A) N^{-1/d}} x\Big) \leqslant p, \;\;\; \forall x\in \R^p.
    $$
    Moreover, the value of $\eta(s,d,A)$ can be taken such that $\lim_{s\to \infty} \eta(s,d,A)^{1/s} = 1$.
\end{customthm}
\begin{zamech}
    The proof of the claim about $\eta(s,d,A)$ can be found in the first reference \cite{reznikovCovering2018}, where it suffices to observe that $ \mathbb R^p$ is a convex set without boundary. 
\end{zamech}
\begin{proof}[Proof of Theorem~\ref{thm:covering_polarized}.]
In this proof, constants $c_1, c_2, c_3$ can depend only on $p$. For brevity, we write
$$
\rho:=\rho^*(A,N), \qquad  \eta:=\eta(s,d,A)
$$
for the optimal $ N $-point covering radius of set $ A $ and the constant $ \eta $ from Theorem~\ref{thm:weak_sep}, respectively.
Let a point $y\in A$ be such that $\rho=\min_j |y-x_j|$. For an integer $n\geqslant 2$, set
$$
H_n:=\B{ n\rho}{y}\setminus \B{(n-1)\rho}y.
$$
Notice that the choice of $y$ implies that $\B{ \rho}{y} \cap \omega_N=\emptyset$, and thus
$
\omega_N \subset \bigcup_{n=2}^\infty H_n.
$
Furthermore, 
$$\L_p(H_n) = v_pn^p\rho^p - v_p(n-1)^p\rho^p \leqslant c_1 v_p n^{p-1} \rho^{p}. $$

We can cover $H_n$ with balls of radius $\eta N^{-1/d}$; a single ball of this radius has volume $ v_p (\eta N^{-1/d})^p$, so at most 
$
c_2\eta^{-p} n^{p-1} \rho^p N^{p/d}
$
such balls are required to cover $ H_n $. In addition, the weak separation property implies that each ball contains no more than $p$ points from $\omega_N$, so that for $n\geqslant 2$ we have
$
\#(H_n \cap \omega_N)\leqslant c_3 \eta^{-p} n^{p-1} \rho^p N^{p/d}.
$
This results in
\[
\begin{aligned}
    P_s(\omega_N,A) 
    &\leqslant \sli_{j=1}^N \frac{1}{|y-x_j|^s} = \sli_{n=2}^\infty \sli_{x_j \in H_n} \frac{1}{|y-x_j|^s} \leqslant \rho^{-s} \sli_{n=2}^\infty (n-1)^{-s} \cdot \#(H_n \cap \omega_N)  \\
    &\leq c_3 \eta^{-p} \rho^{p-s} N^{p/d} \sli_{n=2}^\infty (n-1)^{-s} n^{p-1},
\end{aligned}
\]
where it is used that for $x\in H_n$ we have $|y-x|\geqslant (n-1)\rho$.
The series in the right-hand side converges in view of the assumption $s>p$. 

Notice that
$$
\lim_{s\to \infty} \left(\sli_{n=2}^\infty (n-1)^{-s} n^p\right)^{1/s}=1,
$$
so for the constant
$$
c(s,d,p,A):=c_3 \eta^{-p} \cdot \sli_{n=2}^\infty (n-1)^{-s} n^p
$$
we have $\lim_{s\to \infty} c(s,d,p,A)^{1/s}=1$, and
$$
P(A, \omega_N) \leqslant c(s,d,p,A) \rho^{p-s} N^{p/d}.
$$
Solving for $ \rho $ while taking into account $s>p$ 
$$
\rho \leqslant c(s,d,p,A)^{\frac{1}{s-p}} N^{\frac{p}{d(s-p)}} P(A, \omega_N)^{\frac{1}{p-s}}.
$$
It remains to denote $ C:=c(s,d,p,A)^{\frac{1}{s-p}},$ and the proof is complete.
\end{proof}
\begin{proof}[Proof of Corollary~\ref{cor:limiting_polar}]
    To obtain equalities \eqref{sasha_limits1}--\eqref{sasha_limits2}, we establish the inequalities in both directions.
    Let $\omega_N^*$ be a configuration on which $\rho^*(A, N)$ is attained. Then 
    $$
    \P_s^*(A, N)\geqslant P_s(A, \omega_N^*) = \min_{y\in A} \sli_{j=1}^N \frac{1}{|y-x_j|^s} \geqslant \left(\frac{1}{\rho^*(A, N)}\right)^s,
    $$
    implying immediately \eqref{sasha_limits1}--\eqref{sasha_limits2} with ``$ \geq $'' in place of equality.

    To obtain the converse direction, observe that Theorem~\ref{thm:covering_polarized} implies for $s>p$,
    $$
    \left(\frac{\P_s^*(A, N)}{N^{s/d}}\right)^{1/(s-p)} \leqslant \frac{1}{C(s,d,p,A)}\cdot \frac{1}{\rho^*(A, N) N^{1/d}}.
    $$
    It remains to recall that $C(s,d,p,A)\to 1$ as $s\to \infty$ to conclude that \eqref{sasha_limits1}--\eqref{sasha_limits2} hold with ``$ \leq $''.
\end{proof}

\subsection{Polarization on fractal sets}
To prove Theorems~\ref{sasha_fractal_indep}--\ref{sasha_nonex} we need the following well-known result \cite{lalleyPacking1988,lalleydiscrete,feller}. 
\begin{customthm}{C}[Continuous renewal theorem]
    \label{sasha_renewal_cont}
Let $\mu$ be a probability measure on $[0, \infty)$ and $Z(u)$ be a function defined on $[0,\infty)$. Assume that for some positive constants $C$ and $\ep$ and  $u$ sufficiently large there holds
$$
\left| Z(u) - \ili_0^u Z(u-x)\, d\mu(x) \right| \leqslant C e^{-\ep u}.
$$
Then $ \lim_{u\to \infty} Z(u) $ exists. 
\end{customthm}
\begin{remark}
    In a general formulation of this theorem, the exponential on the right can be replaced with any noninreasing function from $ L^1([0,+\infty), \mathcal L_1) $.
\end{remark}
Observe that $ \mathcal P_s^*(A,N) $ is nondecreasing in $ N $. For $t>0$, we define
$$
N(t):=\min\{N\colon \PP_s^*(A, N) \geqslant t\},
$$
a generalized inverse function to $ \mathcal P_s^* $. Clearly existence of the limit
$
\lim_{t\to \infty} {N(t)}{t^{-d/s}}
$
is equivalent to that of
$
\lim_{N\to \infty} {\PP_s^*(A, N)}/{N^{-s/d}}.
$
\begin{proof}[Proof of Theorem~\ref{sasha_fractal_indep}]
    We cannot apply Theorem~\ref{sasha_renewal_cont} directly to the function $N(t)\cdot t^{-d/s}$; it will be necessary to introduce a change of variables first.
    Consider the quantity
    \[
        Z(u) :=  \left(e^u\right)^{-d/s} N(e^u).
    \]
    In order to apply the renewal theorem to $ Z $, it suffices to show that
    \begin{equation}
        \label{eq:Z_renewal}
            Z(u) = z(u) + \int_0^u Z(u-x)\, d\mu(x),
    \end{equation}
    where $ |z(u)| \leq C e^{-\ep u} $ and $ \mu $ places the weight $ r_m^d $ in $ -s\log r_m $:
    \[
        \mu = \sum_{m=1}^M r_m^d\,\delta_{-s\log r_m}.
    \]
    With such $ \mu $, equation~\eqref{eq:Z_renewal} is equivalent to
    \[
        \begin{aligned}
            \left(e^u\right)^{-d/s} N(e^u) & = z(u) + \sum_{m=1}^M Z(u + s\log r_m)r_m^d\\
                                           & = z(u) + \sum\left(e^u r_m^s\right)^{-d/s} N(e^u r_m^s)r_m^d\\
                                           & = z(u) + (e^u)^{-d/s} \sum_{m=1}^M N(e^u r_m^s),
        \end{aligned}
    \]
    which after setting $ L(e^u) := \left(e^u\right)^{d/s} z(u) $ and $ t= e^u $ becomes
    \begin{equation}
        \label{eq:Nrenewal}
        N(t) = L(t) + \sum_{m=1}^M N(tr_m^s).
    \end{equation}
    Hence, to apply Theorem~\ref{sasha_renewal_cont}, it suffices to show that in \eqref{eq:Nrenewal}, $  |L(t)| \leq Ct^{d/s-\epsilon} $ for sufficiently large $ t $. The rest of the proof establishes the two bounds on $ L(t) $ giving this estimate.

    \medskip

    \noindent{\bf Upper bound.}
    Denote $ N_m = N(tr_m^s)$, $ 1\leq m \leq M $, and
    let $ N = \sum_m N_m $. By definition, $ \P_s^*(A, N_m) \geq tr_m^s, $  $ 1\leq m \leq M $, so that for configurations $ \omega_{N_m}^* $ attaining $ \P_s^*(N_m) $ one has $P_s(\omega_{N_m}^*, A) \geq tr_m^s$. Furthermore, due to the nonnegativity and scale-invariance of the polarization functional there holds 
    \[
        \P_s^*(A,N) \geq P_s\left(\bigcup_m \psi_m(\omega_{N_m}^*),A\right) \geq \min_m \left\{ r_m^{-s}\P_s^*(A, N_m) \right\} \geq t.
    \]
    Thus, $N =  \sum_m N(tr_m^s)  \geq N(t)$, giving
    \[
        L(t) = N(t) - \sum_{m=1}^M N(tr_m^s)  \leq 0.
    \]
    {\bf Lower bound.} 
Let $ N = N(t) $ for some $ t>0 $, so by definition $ \P_s^*(A,N) \geq t $. Let $ A_m:= \psi_m(A) $ for $ 1\leq m \leq M $. Because $ A $ satisfies the open set condition, there exists an $ h>0 $, such that $ \dist(A_l, A_k) > 2h $, $ l\neq k$. Consider an optimal polarizing configuration $ \omega_N^* $; let 
$$ N_m:= \#[\omega_N^* \cap \B h{A_m}]$$
be the number of points from this configuration that lie within $ h $ from $ A_m $,  $ 1\leq m \leq M $. Observe that the contribution of points in $ \omega_N^* \setminus \B h{A_1} $ to the values of $ U(y,\omega_N^*) = \sli_{x\in \omega_N^*}\|y-x\|^{-s} $ with $ y \in A_1 $ is at most $ h^{-s} N $, implying
    \[
        h^{-s}N + \min_{y\in A_1} \sum_{x\in \omega_N^*\cap \B h{A_1}} | y - x|^{-s} \geq \P_s^*(A,N) \geq t.
    \]
    Since $ \omega_N^* \cap \B h{A_1} $ can be rescaled using $ \psi_1^{-1} $ and be used as a polarization configuration for the entire $ A $,
    \begin{equation*}
        \P_s^*(A, N_1) \geq P_s\left(\psi_1^{-1} [\omega_N^*\cap \B h{A_1}] \right) \geq r_1^s \left(t - cN\right)
    \end{equation*}
    for $ c = h^{-s} $. To complete the proof, observe that minimal polarization is superadditive: for any $ n_1,\,n_2 $,
    \[
        \P_s^*(A, n_1+n_2) \geq \P_s^*(A, n_1) + \P_s^*(A, n_2).
    \]
    Pick an $ \tilde N_1 := N(c r_1^s N) \leq C N^{d/s}  $. There holds 
    \[
        \P_s^*(A, N_1 + \tilde N_1 ) \geq \P_s^*(A, N_1) + \P_s^*(A, \tilde N_1) \geq r_1^s \left(t - cN\right) + cr_1^s N = r_1^s t,
    \] 
    which implies $ N( r_1^s t) \leq N_1 + \tilde N_1 \leq N_1 + CN^{d/s} $.

    The above argument applies to all $ A_m $, whereby we have $ N(t r_m^s) \leq N_m + CN^{d/s} $. Recall also that $ N \geq \sum_m N_m $. Combining the last two inequalities we finally have
    \[
        N(t) \geq N(tr_1^s) + N(tr_2^2) - CN^{d/s}.
    \]
    As $ N(t) \asymp t^{d/s} $ by \eqref{eq:asymp_bounds}, this gives $ L(t) \geq -CN^{d/s} \geq - Ct^{d^2/s^2} $; since $ d<s $, we thus obtain the desired lower bound for $ L(t) $, and are in the position to use the renewal theorem. This completes the proof.
\end{proof}

\begin{proof}[Proof of Theorem \ref{sasha_nonex}]
First observe that since the additive group
$$
\{t_1 \log(r_1) + \cdots + t_M\log(r_M)\colon t_1,\ldots,t_M\in \mathbb{Z}\}
$$
is not dense in $\R$, it must have the form $ a\mathbb Z $ for some $ a > 0 $. As a result, there exists a number $r\in (0,1)$ and positive integers $i_1, \ldots, i_M$ such that
$$
r_m = r^{i_m}, \quad m=1,\ldots,M.
$$
Without loss of generality, assume $i_1 \geqslant i_2 \geqslant \ldots \geqslant i_m$.
Similarly to the proof of Theorem~\ref{sasha_fractal_indep}, we set
$$
N(t):=\min\{N\colon \rho(A, N) \leq t\},
$$
and consider the sequence
$$
R_n:=N(r^n), \qquad n\geq 1.
$$
Let  $J$ be an integer for which $ 2 r^J < \min_{l\neq k} \dist(A_l,A_k) $ with $ A_m = \psi_m(A) $. We will show that for some constant $C>1$ and every $n\geqslant J$, there holds
\begin{equation}
    \label{sasha_main_nonex}
    \rho(A, R_n-1) \geqslant C r^n.  
\end{equation}
Since $ \rho(A, R_n) \leq r^n $ by definition, this will be sufficient to justify the inexistence of limits claimed in the theorem.
The first step towards \eqref{sasha_main_nonex} is the renewal equation for covering:
\begin{equation}
    \label{eq:renewal_covering}
    R_{n_0} = N(r^{n_0}) = \sli_{m=1}^m R_{n_0-i_m} = \sli_{m=1}^M N(r^{n_0-i_m}).
\end{equation}
Indeed, for a configuration $ \omega_N^* $ attaining the covering radius $ r^{n_0} $, assume a point $ y $ maximizes the distance to $ \omega_N^* $ in $ A $, 
\[
    \|y-x_j\| = \dist(y,\omega_N^*) = r^{n_0}.
\]
It follows that $ y $ and $ x_j $ must belong to the same set $ A_m $, since $r^{n_0} < \dist(A_l, A_k) $, $ l\neq k $. Hence, covering radius in each of $ A_m $ is at most $ r^{n_0} $, implying $ \#(\omega_N^*\cap A_m) \geq N(r^{n_0-i_m})  $, $ 1\leq m \leq M $. On the other hand, due to the minimality of $ N(t) $, the converse inequality also holds, giving \eqref{eq:renewal_covering}.

To prove \eqref{sasha_main_nonex}, observe that by definition of $ N(t) $,
$$
\rho(A, R_{J+n}-1) > r^{J+n}, \qquad n \geq 1,
$$
and we emphasize that this inequality is strict.  Let 
$$
C:= \min \left\{2,\  \min_{n=0, \ldots, i_1} r^{-n-J} \rho(A, R_{J+n}-1)\right\} > 1.
$$
By construction, \eqref{sasha_main_nonex} is satisfied for $n=J,\ldots, J+i_1$. We proceed to obtain~\eqref{sasha_main_nonex} by induction. 

Assume that $n_0>J+i_1$, and the estimate \eqref{sasha_main_nonex} is known for every $n=J, \ldots, n_0-1$. 
Take a configuration $\omega^*$ of cardinality $ {R_{n_0}-1} $ that is optimal for $\rho(A, R_{n_0}-1)$. By~\eqref{eq:renewal_covering}, one of the sets $ A_m $ contains at most $ R_{n_0-i_m}-1 $ elements of $ \omega^* $; suppose this is the case for $ m=j $. The induction hypothesis applied to $ \omega^* \cap A_j$ shows that
\[
    \rho(A_j, R_{n_0-i_j}-1) \geq r_j \cdot Cr^{n_0-i_j} = Cr^{n_0}.
\]
Also, the elements of $ \omega^* \setminus A_j $ do not contribute to covering on $ A_j $ as $ \dist(A_l,A_k) > 2 r^J \geq Cr^{n_0}  $, implying
\[
    \rho(A, R_{n_0}-1) \geq Cr^{n_0},
\]
which completes the proof of both \eqref{sasha_main_nonex} and the first claim of the theorem. Proof for the unconstrained covering is obtained with minor adjustments in the argument, by using $ \B h{A_m} $ in place of $ A_m $, with $ 2h < \min_{l\neq k} \dist(A_k, A_l) $. Finally, the nonexistence of polarization limits for large $ s $ is a consequence of the above results for covering and Theorem~\ref{thm:covering_polarized}.
\end{proof}

\bibliographystyle{acm}
\bibliography{refs_polarization}

\vskip\baselineskip

{\footnotesize
    {\sc Department of Mathematics, Florida State University, Tallahassee, FL 32306}

    \smallskip
    {\it Email:} {\tt \{aanderso, reznikov, vlasiuk, ewhite\}@math.fsu.edu}

    \vspace{1cm}

    \indent{\sc Current address: Department of Mathematics, Vanderbilt University, Nashville, TN, 37240}

    {\it Email:} {\tt oleksandr.vlasiuk@vanderbilt.edu}
}
\end{document}